\documentclass[11pt,reqno]{amsart}
\usepackage{amsfonts, amssymb, amsmath, enumerate, graphicx, bm, caption, subcaption}
\usepackage[usenames,dvipsnames]{color}
\usepackage[linktoc=page, colorlinks, linkcolor=blue, citecolor=blue]{hyperref} 

\numberwithin{equation}{section}

\DeclareMathOperator{\E}{\mathbb{E}}
\DeclareMathOperator{\Var}{Var}

\DeclareMathOperator*{\tr}{tr}
\DeclareMathOperator*{\rank}{rank}
\DeclareMathOperator{\vol}{vol}
\DeclareMathOperator{\conv}{conv}
\DeclareMathOperator{\diam}{diam}

\DeclareMathOperator{\Unif}{Unif}
\DeclareMathOperator{\sign}{sign}

\renewcommand{\Pr}[2][]{\mathbb{P}_{#1} \left\{ #2 \rule{0mm}{3mm}\right\}}
\newcommand{\ip}[2]{\left\langle#1,#2\right\rangle}

\def \R {\mathbb{R}}

\def \CC {\mathcal{C}}
\def \FF {\mathcal{F}}

\def \a {\alpha}

\def \e {\varepsilon}
\def \d {\delta}
\def \l {\lambda}
\def \s {\sigma}

\def \tran {\mathsf{T}}

\def \< {\langle}
\def \> {\rangle}

\def \va {\bm{a}}
\def \vb {\bm{b}}
\def \vd {\bm{d}}
\def \ve {\bm{e}}
\def \vg {\bm{g}}
\def \vh {\bm{h}}
\def \vs {\bm{s}}
\def \vt {\bm{t}}
\def \vx {\bm{x}}
\def \vu {\bm{u}}
\def \vv {\bm{v}}
\def \vy {\bm{y}}
\def \vz {\bm{z}}
\def \vX {\bm{X}}

\def \valpha {\bm{\alpha}}
\def \vnu {\bm{\nu}}

\def \veta {\bm{\eta}}

\def \xhat {\widehat{\vx}}
\def \xlin {\widehat{\vx}_{\mathrm{lin}}}
\def \Xhat {\widehat{X}}

\def \vbeta {\bm{\beta}}
\def \vbetahat {\widehat{\vbeta}}

\def \psitwo {{\psi_2}}

\newtheorem{theorem}{Theorem}[section]
\newtheorem{proposition}[theorem]{Proposition}
\newtheorem{corollary}[theorem]{Corollary}
\newtheorem{lemma}[theorem]{Lemma}

\newtheorem{definition}[theorem]{Definition}
\newtheorem{question}[theorem]{Question}

\theoremstyle{remark}
\newtheorem{remark}[theorem]{Remark}
\newtheorem{example}[theorem]{Example}

\title[]{Estimation in high dimensions: \\ a geometric perspective}

\author{Roman Vershynin}

\date{\today}
\address{Department of Mathematics, University of Michigan, 530 Church St., Ann Arbor, MI 48109, U.S.A.}
\email{romanv@umich.edu}
\thanks{Partially supported by NSF grant DMS 1265782 and USAF Grant FA9550-14-1-0009.}

\begin{document}

\begin{abstract}
This tutorial provides an exposition of a flexible geometric framework 
for high dimensional estimation problems with constraints.
The tutorial develops geometric intuition about high dimensional sets, justifies it with 
some results of asymptotic convex geometry, and demonstrates connections
between geometric results and estimation problems. The theory is illustrated
with applications to sparse recovery, matrix completion, quantization, 
linear and logistic regression and generalized linear models.
\end{abstract}

\setcounter{tocdepth}{1}

\maketitle

\tableofcontents

\newpage

\section{Introduction}

\subsection{Estimation with constraints}
This chapter provides an exposition of an emerging mathematical framework for 
high-dimensional {\em estimation problems with constraints}.
In these problems, the goal is to estimate a point $\vx$ which lies
in a certain known feasible set $K \subseteq \R^n$, from a small sample
$y_1,\ldots, y_m$ of independent observations of $\vx$.
The point $\vx$ may represent a signal in signal processing, 
a parameter of a distribution in statistics, or an unknown matrix 
in problems of matrix estimation or completion.
The feasible set $K$ is supposed to represent properties that we know 
or want to impose on $\vx$. 

The geometry of the high dimensional set $K$ is a key to understanding 
estimation problems. A powerful intuition about {\em what high dimensional 
sets look like} has been developed in the area known as {\em asymptotic convex geometry}
\cite{Ball, GBVV}.
The intuition is supported by many rigorous results, some of which can be 
applied to estimation problems. The main goals of this chapter are:

\begin{enumerate}[\quad (a)]
  \item develop geometric intuition about high dimensional sets;
  \item explain results of asymptotic convex geometry which validate
    this intuition;
  \item demonstrate connections between high dimensional geometry
	and high dimensional estimation problems.
\end{enumerate}

This chapter is not a comprehensive survey but is rather a tutorial.
It does not attempt to chart vast territories of high dimensional inference
that lie on the interface of statistics and signal processing.
Instead, this chapter proposes a useful geometric viewpoint, 
which could help us find a common mathematical 
ground for many (and often dissimilar) estimation problems.

\subsection{Quick examples}

Before we proceed with a general theory, let us mention 
some concrete examples of estimation problems that will be covered here.
A particular class of estimation problems with constraints 
is considered in the young field of {\em compressed sensing} \cite{DDEK, Kutyniok, CGLP, FR}.
There $K$ is supposed to enforce {\em sparsity}, thus $K$ usually consists of vectors
that have few non-zero coefficients. Sometimes more restrictive
{\em structured sparsity} assumptions are placed, where only certain arrangements
of non-zero coefficients are allowed \cite{BJMO, RRN}. The observations $y_i$ in compressed
sensing are assumed to be {\em linear} in $\vx$, which means that 
$y_i = \ip{\va_i}{\vx}$. Here $\va_i$ are typically i.i.d. vectors drawn from 
some known distribution in $\R^n$ (for example, normal). 

Another example of estimation problems with constraints 
is the {\em matrix completion problem}
\cite{CR, CT, KMO, Gross, SAT, Recht} where $K$ consists of matrices with {\em low rank}, and 
$y_1,\ldots,y_m$ is a sample of matrix entries. Such observations 
are still linear in $\vx$. 

In general, observations do not have to be linear; 
good examples are {\em binary observations} $y_i \in \{-1,1\}$, which satisfy
$y_i = \sign(\ip{\va_i}{\vx})$, see \cite{BB, JLBB, PV CPAM, PV DCG},
and more generally $\E y_i = \theta(\ip{\va_i}{\vx})$, see \cite{PV IEEE, ALPV, PVY}. 

In statistics, these classes of estimation problems can be interpreted
as {\em linear regression} (for linear observations with noise), 
{\em logistic regression} (for binary observations)
and {\em generalized linear models} (for more general non-linear observations). 

All these examples, and more, will be explored in this chapter. 
However, our main goal is to advance a general approach, which  
would not be tied to a particular nature of the feasible set $K$.
Some general estimation problems of this nature were considered in \cite{MPT, ALMT} 
for linear observations and in \cite{PV IEEE, PV DCG, ALPV, PVY} for non-linear observations.

\subsection{Plan of the chapter}

In Seciton~\ref{s: setup}, we introduce a general class of estimation problems
with constraints. We explain how the constraints (given by feasible set $K$) 
represent {\em low-complexity structures}, which could make it possible
to estimate $\vx$ from few observations. 

In Section~\ref{s: geometry}, 
we make a short excursion into the field of {\em asymptotic convex geometry}.
We explain intuitively the shape of high-dimensional sets $K$
and state some known results supporting this intuition.
In view of estimation problems, we especially emphasize one of these
results -- the so-called {\em $M^*$ bound}
on the size of high-dimensional sections of $K$ by a random subspace $E$.
It depends on the single geometric parameter of $K$ that quantifies the 
complexity of $K$; this quantity is called the {\em mean width}. 
We discuss mean width in some detail, 
pointing out its connections to convex geometry, stochastic processes, 
and statistical learning theory. 

In Section~\ref{s: estimation linear} we apply the $M^*$ bound to 
the general estimation problem with linear observations. 
We formulate an estimator first as a convex feasibility problem 
(following \cite{MPT}) and then as a convex optimization problem.

In Section~\ref{s: M* proof} we prove a general form of the $M^*$ 
bound. Our proof borrowed from \cite{PV DCG} is quite simple and instructive. 
Once the $M^*$ bound is stated in the language of stochastic processes, 
it follows quickly by application of symmetrization, contraction and rotation invariance.

In Section~\ref{s: noisy}, we apply the general $M^*$ bound to estimation 
problems; observations here are still linear but can be noisy.
Examples of such problems include {\em sparse recovery} problems 
and {\em linear regression} with constraints, which we explore in 
Section~\ref{s: sparse recovery}.

In Section~\ref{s: sub-gaussian}, we extend the theory from Gaussian 
to sub-gaussian observations. A sub-gaussian $M^*$ bound 
(similar to the one obtained in \cite{MPT}) is deduced from 
the previous (Gaussian) argument followed by an application of 
a deep comparison theorem of X.~Fernique and M.~Talagrand 
(see \cite{Talagrand}). 

In Section~\ref{s: exact} we pass to {\em exact recovery} results, 
where an unknown vector $\vx$ can be inferred from the 
observations $y_i$ without any error.
We present a simple geometric argument based on Y.~Gordon's 
``escape through a mesh'' theorem \cite{Gordon}. This argument was first used in this context
for sets of sparse vectors in \cite{RV CPAM}, was further developed in \cite{Stojnic, Oymak-Hassibi}
and pushed forward for general feasible sets in \cite{CRPW, ALMT, Tropp}.

In Section~\ref{s: matrix}, we explore matrix estimation problems. 
We first show how the general theory applies to a {\em low-rank matrix recovery} 
problem. Then we address a {\em matrix completion} problem with 
a short and self-contained argument from \cite{PVY}.

Finally, we pass to non-linear observations. 
In Section~\ref{s: single bit tess}, we consider {\em single-bit observations} 
$y_i = \sign\ip{\va_i}{\vx}$. Analogously to linear observations, 
there is a clear geometric interpretation for these as well. 
Namely, the estimation problem reduces in this case to a 
{\em pizza cutting problem} about random hyperplane tessellations of $K$. 
We discuss a result from \cite{PV DCG} on this problem, and we apply it to
estimation by formulating it as a feasibility problem. 

Similarly to what we did for linear observations, we replace the feasibility problem 
by optimization problem in Section~\ref{s: single bit opt}. Unlike before, such replacement 
is not trivial. We present a simple and self-contained argument from \cite{PV IEEE} about 
estimation from single-bit observations via convex optimization. 

In Section~\ref{s: general non-linear} we discuss the estimation problem 
for general (not only single-bit) observations following \cite{PVY}. 
The new crucial step of estimation is the metric projection onto the feasible set; 
this projection was studied recently in \cite{Chatterjee} and \cite{PVY}.

In Section~\ref{s: extensions}, we outline some natural extensions of the results
for general distributions and to a localized version of mean width.

\subsection{Acknowledgements}
The author is grateful to Vladimir Koltchinskii, Shahar Mendelson, Renato Negrinho, 
Robert Nowak, Yaniv Plan, Elizaveta Rebrova, Joel Tropp and especially the anonymous referees
for their helpful discussions, comments, and corrections, which 
lead to a better presentation of this chapter.

\section{High dimensional estimation problems}			\label{s: hde problems}

\subsection{Estimating vectors from random observations}			\label{s: setup}

Suppose we want to estimate an unknown vector $\vx \in \R^n$.
In signal processing, $\vx$ could be a signal to be reconstructed, while 
in statistics $\vx$ may represent a parameter of a distribution.
We assume that information about $\vx$ comes from a sample 
of independent and identically distributed observations $y_1, \ldots, y_m \in \R$,
which are drawn from a certain distribution which depends on $\vx$:
$$
y_i \sim \text{distribution} (\vx), \quad i=1,\ldots,m.
$$
So, we want to estimate $\vx \in \R^n$ from the observation vector
$$
\vy = (y_1,\ldots,y_m) \in \R^m.
$$
One example of this situation is the classical linear regression problem in statistics, 
\begin{equation}         \label{eq: linear regression}
\vy = X \vbeta + \vnu,
\end{equation}
in which one wants to estimate the coefficient vector $\vbeta$ from the observation 
vector $\vy$. We will see 
many more examples later; for now let us continue with setting up 
the general mathematical framework.

\subsection{Low complexity structures}

It often happens that we know in advance, believe in, or want to enforce,
some properties of the vector $\vx$.
We can formalize such extra information as the assumption that 
$$
\vx \in K
$$
where $K$ is some fixed and known subset of $\R^n$, a {\em feasible set}.
This is a very general and flexible assumption. At this point, we are not stipulating 
any properties of the feasible set $K$. 

To give a quick example, in the regression problem \eqref{eq: linear regression},
one often believes that $\vbeta$ 
is a sparse vector, i.e. among its coefficients only few are non-zero. 
This is important because it means that a few explanatory variables can adequately explain 
the dependent variable. So one could choose $K$ to be a set of all $s$-sparse vectors in $\R^n$ --
those with at most $s$ non-zero coordinates, for a fixed sparsity level $s \le n$.
More examples of natural feasible sets $K$ will be given later. 

Figure~\ref{fig: sampling-estimation} illustrates the estimation problem. 
Sampling can be thought of as a map taking $\vx \in K$ to $\vy \in \R^m$; 
estimation is a map from $\vy \in \R^m$ to $\xhat \in K$ and is ideally the inverse of sampling. 

\begin{figure}[htp]			
  \centering \includegraphics[height=2.6cm]{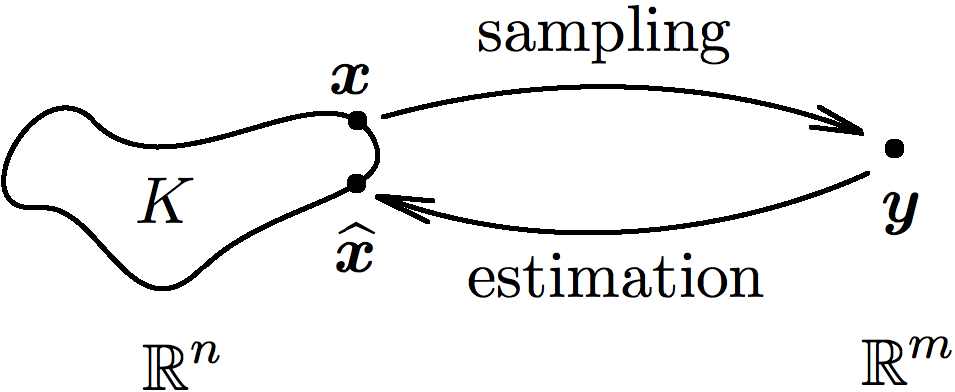} 
  \caption{Estimation problem in high dimensions}
  \label{fig: sampling-estimation}	
\end{figure}

How can a prior information encoded by $K$ help in high-dimensional estimation?
Let us start with a quick and non-rigorous argument based on the number of degrees of freedom. 
The unknown vector $\vx$ has $n$ dimensions
and the observation vector $\vy$ has $m$ dimensions. So in principle, it should 
be possible to estimate $\vx$ from $\vy$ with 
$$
m = O(n)
$$ 
observations. Moreover, this bound should be tight in general. 

Now let us add the restriction that $\vx \in K$. If $K$ happens to be {\em low-dimensional}, 
with algebraic dimension $\dim(K) = d \ll n$, then $\vx$ has $d$ degrees of freedom. 
Therefore, in this case the estimation should be possible with fewer observations, 
$$
m = O(d) = o(n).
$$

It rarely happens that feasible sets of interest literally have small algebraic dimension.
For example, the set of all $s$-sparse vectors in $\R^n$ has full dimension $n$.
Nevertheless, the intuition about low-dimensionality remains valid. 
Natural feasible sets, such as regression coefficient vectors, 
images, adjacency matrices of networks, do tend to have {\em low complexity}. 
Formally $K$ may live in an $n$-dimensional space where $n$ can be very large, 
but the actual complexity of $K$, or ``effective dimension'' 
(which will formally quantify in Section~\ref{s: eff dim}) is often much smaller. 

\medskip

This intuition motivates the following three goals, which we will discuss in detail in this chapter: 

\begin{enumerate}[\qquad 1.]
  \item Quantify the complexity of general subsets $K$ of $\R^n$.
  \item Demonstrate that  estimation can be done with few observations
    as long as the feasible set $K$ has low complexity.
  \item Design estimators that are algorithmically efficient.
\end{enumerate}

We will start by developing intuition about the geometry of sets $K$ in high dimensions.
This will take us a short excursion into high dimensional convex geometry.
Although convexity assumption for $K$ will not be imposed in most results of this chapter, 
it is going to be useful in Section~\ref{s: geometry} for developing a good intuition 
about geometry in high dimensions.

\section{An excursion into high dimensional convex geometry}		\label{s: geometry}

High dimensional convex geometry studies {\em convex bodies} $K$ in $\R^n$ for large $n$; 
those are closed, bounded, convex sets with non-empty interior. 
This area of mathematics is sometimes also called asymptotic convex geometry (referring 
to $n$ increasing to infinity) and geometric functional analysis. 
The tutorial \cite{Ball} could be an excellent first contact with this field;  
the survey \cite{GM} and books \cite{MS, Pisier, GBVV, AGM} cover more material and in more depth.

\subsection{What do high dimensional convex bodies look like?}

A central problem in high dimensional convex geometry is -- {\em what do convex 
bodies look like in high dimensions?} A heuristic answer to this question is --
a convex body $K$ usually consists of a {\em bulk} and {\em outliers}. 
The bulk makes up most of the volume of $K$, but it is usually small in diameter. 
The outliers contribute little to the volume, but they are large in diameter. 

If $K$ is properly scaled, the bulk usually looks like a Euclidean ball.
The outliers look like thin, long tentacles. This is best seen on Figure~\ref{fig: general convex}, which 
depicts V.~Milman's vision of high dimensional convex sets \cite{Milman}. This picture does 
not look convex, and there is a good reason for this. The volume in high dimensions 
scales differently than in low dimensions -- dilating of a set by the factor $2$ 
increases its volume by the factor $2^n$. This is why it is not surprising that 
the tentacles contain exponentially less volume than the bulk. Such behavior
is best seen if a picture looks ``hyperbolic''. Although not convex, pictures like 
Figure~\ref{fig: milman-convex-body}	 more accurately reflect the distribution of 
volume in higher dimensions.

\begin{figure}[htp]			
  \centering 
  \begin{subfigure}[b]{0.4\textwidth}
    \includegraphics[height=3.4cm]{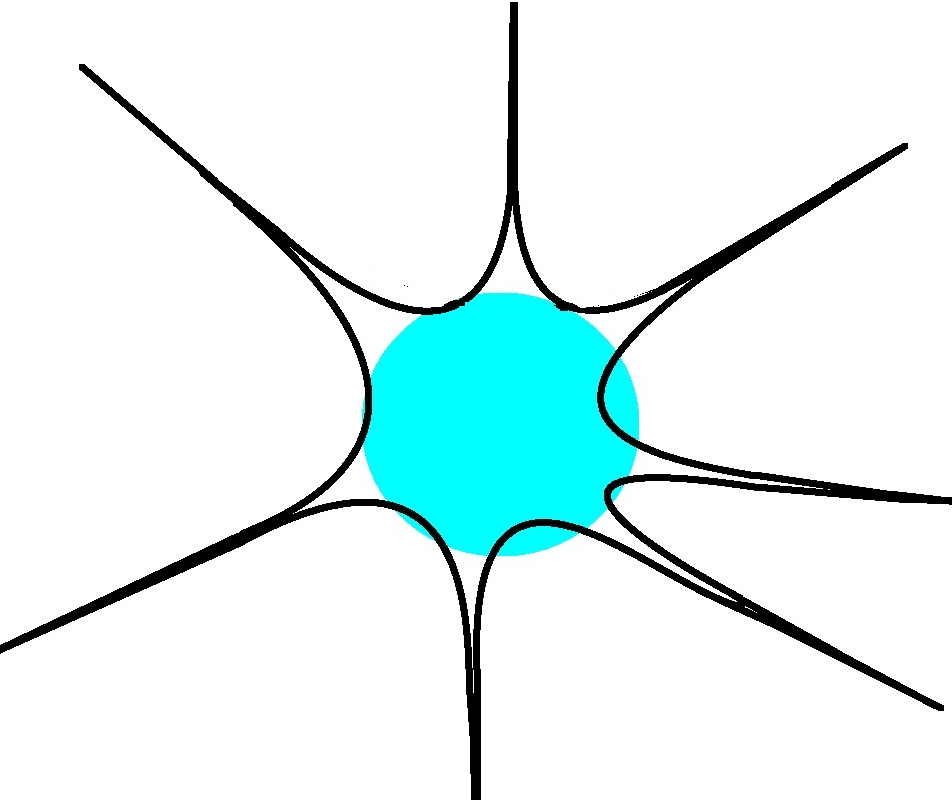} 
    \caption{A general convex set}
    \label{fig: general convex}
  \end{subfigure}
  \qquad \qquad
  \begin{subfigure}[b]{0.3\textwidth}
    \includegraphics[height=3.4cm]{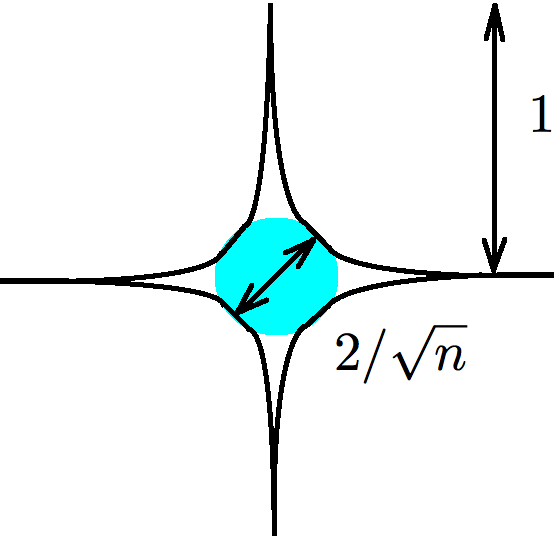} 
    \caption{The $\ell_1$ ball}
    \label{fig: ell1}
  \end{subfigure}
  \caption{V.~Milman's ``hyperbolic'' drawings of high dimensional convex sets }
  \label{fig: milman-convex-body}	
\end{figure}

\begin{example}[The $\ell_1$ ball]
  To illustrate this heuristic on a concrete example, consider the set
  $$
  K = B_1^n = \{ x \in \R^n :\; \|x\|_1 \le 1 \},
  $$ 
  i.e. the unit $\ell_1$-ball in $\R^n$. The inscribed Euclidean ball in $K$, 
  which we will denote by $B$, has diameter $2/\sqrt{n}$. 
  One can then check that volumes of $B$ and of $K$ are comparable:\footnote{Here $a_n \asymp b_n$
  means that there exists positive absolute constants $c$ and $C$ such that 
  $c a_n \le b_n \le C a_n$ for all $n$.}
  $$
  \vol_n(B)^{1/n} \asymp \vol_n(K)^{1/n} \asymp \frac{1}{n}.
  $$
  Therefore, $B$ (perhaps inflated by a constant factor) forms 
  the bulk of $K$. It is round, makes up most of the volume of $K$, but has small diameter. 
  The outliers of $K$ are thin and long tentacles protruding quite far in the coordinate directions. 
  This can be best seen in a hyperbolic drawing, see Figure~\ref{fig: ell1}.
\end{example}

\subsection{Concentration of volume}

The heuristic representation of convex bodies just described 
can be supported by some rigorous results about 
{\em concentration of volume}. 

These results assume that $K$ is {\em isotropic}, which means that 
the random vector $X$ distributed uniformly in $K$ (according to the Lebesgue measure) 
has zero mean and identity covariance: 
\begin{equation}         \label{eq: isotropy}
\E X = 0, \quad \E X X^\tran = I_n.
\end{equation}
Isotropy is just an assumption of proper scaling -- one can always make a convex body $K$ 
isotropic by applying a suitable invertible linear transformation. 

With this scaling, most of the volume of $K$ is located around the 
Euclidean sphere of radius $\sqrt{n}$. Indeed, taking traces of both sides 
of the second equation in \eqref{eq: isotropy}, we obtain 
$$
\E \|X\|_2^2 = n.
$$
Therefore, by Markov's inequality, at least $90\%$ of the volume of $K$ 
is contained in a Euclidean ball of size $O(\sqrt{n})$. 
Much more powerful concentration results are known -- the bulk of 
$K$ lies very near the sphere of radius $\sqrt{n}$, and the outliers 
have exponentially small volume. This is the content of the two major
results in high dimensional convex geometry, which we summarize in the 
following theorem. 

\begin{theorem}[Distribution of volume in high-dimensional convex sets]				\label{thm: volume}
  Let $K$ be an isotropic convex body in $\R^n$, and let $X$ be a random vector
  uniformly distributed in $K$. Then the following is true: 
  \begin{enumerate}[\qquad 1.]
    \item(Concentration of volume) For every $t \ge 1$, one has 
      $$
      \Pr{\|X\|_2 > t \sqrt{n}} \le \exp(-c t \sqrt{n}). 
      $$
    \item(Thin shell) For every $\e \in (0,1)$, one has
      $$
      \Pr{\Big| \|X\|_2 - \sqrt{n} \Big| > \e \sqrt{n}} \le C \exp(-c \e^3 n^{1/2}).
      $$
  \end{enumerate} 
  Here and later in this chapter, $C, c$ denote positive absolute constants.
\end{theorem}

The concentration part of Theorem~\ref{thm: volume} is due to G.~Paouris \cite{Paouris}; 
see \cite{ALLOPT} for an alternative and shorter proof. The thin shell part is an improved 
version of a result of B.~Klartag \cite{Klartag}, which is due 
to O.~Guedon and E.~Milman \cite{Guedon-Milman}.

\subsection{Low dimensional random sections}

The intuition about bulk and outliers of high dimensional convex bodies $K$ can help us to understand
what {\em random sections} of $K$ should look like. Suppose $E$ is a random subspace of $\R^n$
with fixed dimension $d$, i.e. 
 $E$ is drawn at random from the Grassmanian manifold $G_{n,d}$ according to the Haar measure. 
What does the section $K \cap E$ look like on average? 

If $d$ is sufficiently small, then we should expect $E$ to pass through the bulk of $K$ and
miss the outliers, as those have very small volume. 
Thus, if the bulk of $K$ is a round ball,\footnote{This intuition is a good approximation to truth, 
  but it should to be corrected. 
  While concentration of volume tells us that the bulk is {\em contained} in a certain Euclidean ball
  (and even in a thin spherical shell), 
  it is not always true that the bulk {\em is} a Euclidean ball (or shell); a counterexample is 
  the unit cube $[-1,1]^n$. 
  In fact, the cube is the worst convex set in the Dvoretzky theorem, which we are about to state.}  
we should expect the section $K \cap E$ to be a {\em round ball} 
as well; see Figure~\ref{fig: milman-convex-body-section}.

\begin{figure}[htp]			
  \centering 
  \includegraphics[height=3.4cm]{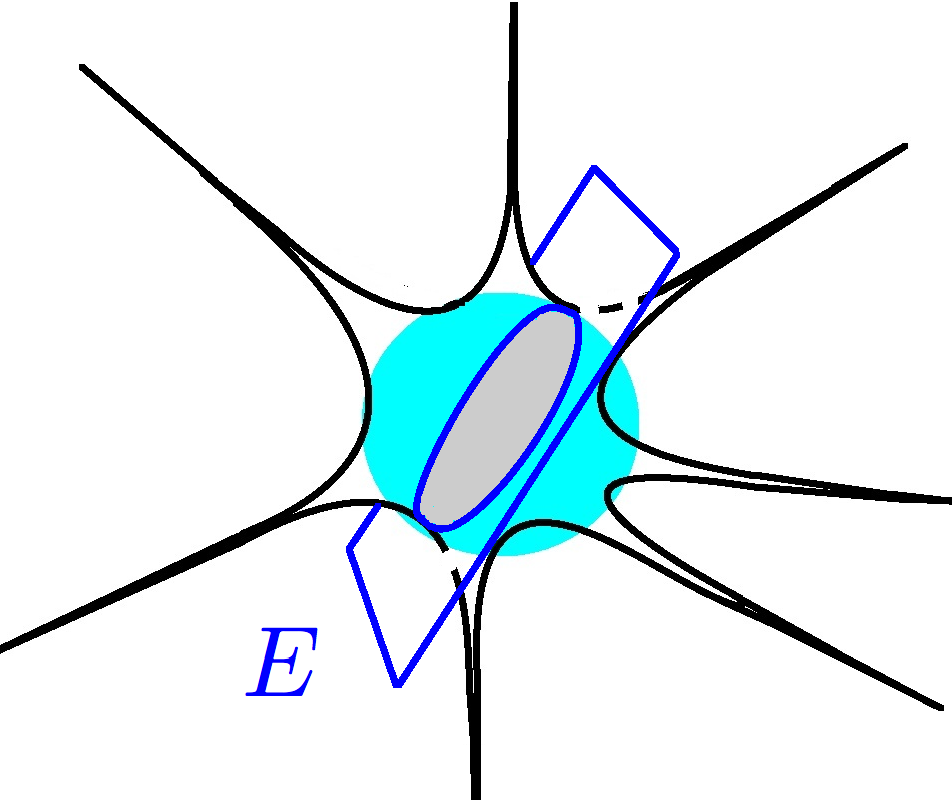} 
  \caption{Random section of a high dimensional convex set}
  \label{fig: milman-convex-body-section}	
\end{figure}

There is a rigorous result which confirms this intuition. It is known as 
Dvoretzky's theorem \cite{Dvoretzky1, Dvoretzky2}, 
which we shall state in the form of V. Milman \cite{Milman-Dvoretzky}; 
expositions of this result can be found e.g. in \cite{Pisier, GBVV}. 
Dvoretzky-Milman's theorem has laid a foundation for 
the early development of asymptotic convex geometry. 
Informally, this result says that 
random sections of $K$ of dimension $d \sim \log n$ are round with high probability. 

\begin{theorem}[Dvoretzky's theorem]					\label{thm: dvoretzky}
  Let $K$ be an origin-symmetric convex body in $\R^n$ such that the ellipsoid of maximal volume 
  contained in $K$ is the unit Euclidean ball $B_2^n$. Fix $\e \in (0,1)$.  
  Let $E$ be a random subspace of dimension 
  $d = c\e^{-2} \log n$ drawn from the Grassmanian $G_{n,d}$ according to the Haar measure. 
  Then there exists $R \ge 0$ such that with high probability (say, $0.99$) we have
  $$
  (1-\e) \, B(R) \subseteq K \cap E \subseteq (1+\e) \, B(R).
  $$
  Here $B(R)$ is the centered Euclidean ball of radius $R$ in the subspace $E$.
\end{theorem}

Several important aspects of this theorem are not mentioned here -- in particular how, 
for a given convex set $K$, to compute the radius $R$ and the largest dimension $d$ of round sections 
of $K$. These aspects can be found in modern treatments of Dvoretzky theorem
such as \cite{Pisier, GBVV}.

\subsection{High dimensional random sections?}		\label{s: high dim sections?}

Dvoretzky's Theorem~\ref{thm: dvoretzky} describes the shape of {\em low} dimensional random
sections $K \cap E$, those of dimensions $d \sim \log n$. 
Can anything be said about {\em high} dimensional sections, those with small codimension?  
In this more difficult regime, we can no longer expect such sections to be round.
Instead, as the codimension decreases, the random subspace $E$ becomes larger and it will probably
pick more and more of the outliers (tentacles) of $K$. The shape of such sections $K \cap E$ 
is difficult to describe. 

Nevertheless, it turns out that we can accurately predict the {\em diameter} 
of $K \cap E$. A bound on the diameter is known in asymptotic convex geometry as the low $M^*$ estimate, 
or $M^*$ bound. We will state this result in Section~\ref{s: M*} 
and prove it in Section~\ref{s: M* proof}. 
For now, let us only mention that $M^*$ bound is particularly attractive in applications 
as it depends only on two parameters -- the codimension of $E$ 
and a single geometric quantity, which informally speaking, measures the size of the bulk of $K$. 
This geometric quantity is called the {\em mean width} of $K$. We will pause briefly to discuss this
important notion.

\subsection{Mean width}

The concept of mean width captures important geometric characteristics of sets in $\R^n$.
One can mentally place it in the same category as other classical geometric quantities 
like volume and surface area.

Consider a bounded subset $K$ in $\R^n$. 
(The convexity, closedness and nonempty interior will not be imposed from now on.)
The {\em width} of $K$ in the direction of a given unit vector $\veta \in S^{n-1}$ is defined as the width of the smallest slab 
between two parallel hyperplanes with normals $\veta$ that contains $K$; see Figure~\ref{fig: mean width}. 
\begin{figure}[htp]		
  \centering \includegraphics[height=3cm]{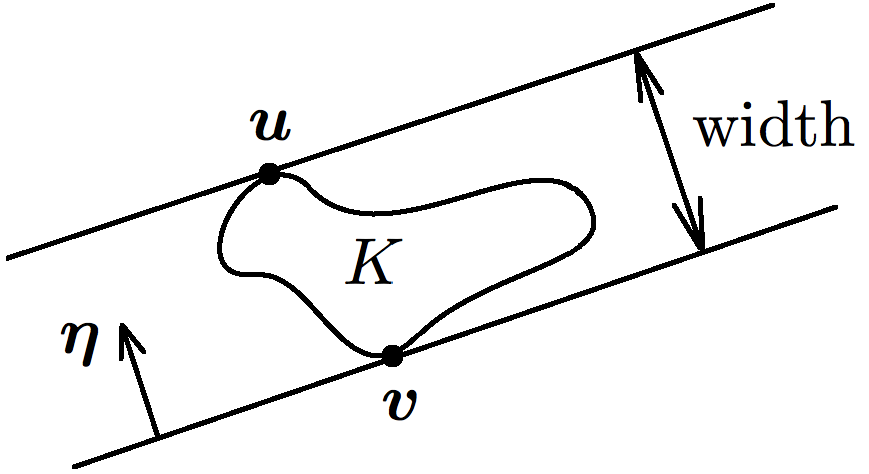} 
  \caption{Width of $K$ in the direction of $\veta$}
  \label{fig: mean width}
\end{figure}

Analytically, we can express the width in the direction of $\veta$ as  
$$
\sup_{\vu, \vv \in K} \ip{\veta}{\vu-\vv} 
= \sup_{\vz \in K-K} \ip{\veta}{\vz}
$$ 
where $K-K = \{\vu-\vv :\; \vu, \vv \in K\}$ is the Minkowski sum of $K$ and $-K$.
Equivalently, we can define the width using the standard notion of {\em support function} 
of $K$, which is $h_K(\veta) = \sup_{\vu \in K} \ip{\veta}{\vu}$, see \cite{Rockafellar}. 
The width of $K$ in the direction of $\veta$ can be expressed as $h_K(\veta) + h_{K}(-\veta)$.

Averaging over $\veta$ uniformly distributed on the sphere $S^{n-1}$, 
we can define the {\em spherical mean width} of $K$:
$$
\tilde{w}(K) := \E \sup_{\vz \in K-K} \ip{\veta}{\vz}.
$$
This notion is standard in asymptotic geometric analysis. 

In other related areas, such as high dimensional probability and 
statistical learning theory, it is more convenient to replace the {\em spherical} 
random vector $\veta \sim \Unif(S^{n-1})$
by the standard {\em Gaussian} random vector $\vg \sim N(0, I_n)$.
The advantage is that $\vg$ has independent coordinates while $\veta$ does not. 

\begin{definition}[Gaussian mean width]				\label{def: mean width}
  The Gaussian {\em mean width} of a bounded subset $K$ of $\R^n$ is defined as 
  \begin{equation}         \label{eq: mean width}
  w(K) := \E \sup_{\vu \in K-K} \ip{\vg}{\vu}, 
  \end{equation}
  where $\vg \sim N(0,I_n)$ is a standard Gaussian random vector in $\R^n$.
  We will often refer to Gaussian mean width as simply the {\em mean width}.
\end{definition}

\subsubsection{Simple properties of mean width}			\label{s: mean width prop}

Observe first 
that the Gaussian mean width is about $\sqrt{n}$ times larger than the spherical mean width. 
  To see this, using rotation invariance we realize $\veta$ as $\veta = \vg/\|\vg\|_2$. 
  Next, we recall that the direction and magnitude of a standard Gaussian random vector are independent, 
  so $\veta$ is independent of $\|\vg\|_2$. It follows that 
  $$
  w(K) = \E \|\vg\|_2 \cdot \tilde{w}(K).
  $$ 
  Further, the factor $\E \|\vg\|_2$ is of order $\sqrt{n}$; this follows, for example, 
  from known bounds on the  $\chi^2$ distribution:
  \begin{equation}         \label{eq: gaussian norm}
  c \sqrt{n} \le \E \|\vg\|_2 \le \sqrt{n}
  \end{equation}
  where $c>0$ is an absolute constant.
Therefore, the Gaussian and spherical versions of mean width 
are equivalent (up to scaling factor $\sqrt{n}$), 
so it is mostly a matter of personal preference which version to work with.
In this chapter, we will mostly work with the Gaussian version. 

Let us observe a few standard and useful properties of the mean width, 
which follow quickly from its definition. 

\begin{proposition}			\label{prop: mean width}
  The mean width is invariant under translations, orthogonal transformations, and taking convex hulls. \qed
\end{proposition}

Especially useful for us will be the last property, which states that 
\begin{equation}         \label{eq: conv wK}
w(\conv(K)) = w(K).
\end{equation}
This property will come handy later, 
when we consider convex relaxations of optimization problems. 

\subsubsection{Computing mean width on examples}

Let us illustrate the notion of mean width on some simple examples.

\begin{example}
If $K$ is {\em the unit Euclidean ball} $B_2^n$ or sphere $S^{n-1}$, then 
$$
w(K) = \E \|\vg\|_2 \le \sqrt{n}
$$
and also $w(K) \ge c \sqrt{n}$, by \eqref{eq: gaussian norm}.
\end{example}

\begin{example}				\label{ex: algebraic dim}
Let $K$ be a subset of $B_2^n$ with {\em linear algebraic dimension $d$}.
Then $K$ lies in a $d$-dimensional unit Euclidean ball, so as before we have 
$$
w(K) \le 2\sqrt{d}.
$$
\end{example}

\begin{example}		\label{ex: finite set mean width}
Let $K$ be {\em a finite subset} of $B_2^n$. Then
$$
w(K) \le C\sqrt{\log|K|}.
$$ 
This follows from a known and simple computation of the expected maximum of $k=|K|$ 
Gaussian random variables.
\end{example}

\begin{example}[Sparsity]				\label{ex: sparse mean width}
Let $K$ consist of all unit $s$-sparse vectors in $\R^n$
-- those with at most $s$ non-zero coordinates:
$$
K = \{ \vx \in \R^n :\; \|\vx\|_2 = 1, \; \|\vx\|_0 \le s \}.
$$
Here $\|\vx\|_0$ denotes the number of non-zero coordinates of $\vx$.
A simple computation (see e.g. \cite[Lemma~2.3]{PV IEEE}) shows that 
$$
c \sqrt{s \log(2n/s)} \le w(K) \le C \sqrt{s \log(2n/s)}.
$$
\end{example}

\begin{example}[Low rank]				\label{ex: low rank mean width}
Let $K$ consist of $d_1 \times d_2$ matrices with unit Frobenius norm and  
rank at most $r$:
$$
K = \{ X \in \R^{d_1 \times d_2} :\; \|X\|_F = 1, \; \rank(X) \le r \}.
$$
We will see in Proposition~\ref{prop: low rank mean width},
$$
w(K) \le C \sqrt{r(d_1 + d_2)}.
$$
\end{example}

\subsubsection{Computing mean width algorithmically}

Can we estimate the mean width of a given set $K$ fast and accurately?
Gaussian concentration of measure 
(see \cite{Pisier, LT, Ledoux}) implies that, with high probability,
the random variable 
$$
w(K,\vg) = \sup_{\vu \in K-K} \ip{\vg}{\vu}
$$
is close to its expectation $w(K)$.
Therefore, to estimate $w(K)$, it is enough to generate a single 
realization of a random vector $\vg \sim N(0,I_n)$ and compute $w(K, \vg)$; 
this should produce a good estimator of $w(K)$.

Since we can convexify $K$ without changing the mean width by Proposition~\ref{prop: mean width}, 
computing this estimator is a {\em convex optimization problem} 
(and often even a linear problem if $K$ is a polytope). 

\subsubsection{Computing mean width theoretically}

Finding theoretical estimates on the mean width of a given set $K$ 
is a non-trivial problem. It has been extensively studied 
in the areas of probability in Banach spaces and stochastic processes.

Two classical results in the theory of stochastic processes -- 
Sudakov's inequality (see \cite[Theorem~3.18]{LT}) 
and Dudley's inequality (see \cite[Theorem~11.17]{LT}) -- relate the 
mean width to the metric entropy of $K$. Let $N(K, t)$ denote 
the smallest number of Euclidean balls of radius $t$ whose union covers $K$. 
Usually $N(K,t)$ is referred to as a {\em covering number} of $K$, 
and $\log N(K,t)$ is called the {\em metric entropy} of $K$. 

\begin{theorem}[Sudakov's and Dudley's inequalities]
  For any bounded subset $K$ of $\R^n$, we have 
  $$
  c \, \sup_{t>0} t \sqrt{\log N(K,t)} 
  \le w(K)
  \le C \int_0^\infty \sqrt{\log N(K,t)} \; dt.
  $$
The lower bound is Sudakov's inequality and the upper bound
is Dudley's inequality.
\end{theorem}

Neither Sudakov's nor Dudley's inequality are tight for all sets $K$. 
A more advanced method of {\em generic chaining} produces 
a tight (but also more complicated) estimate of the mean width
in terms of {\em majorizing measures}; see \cite{Talagrand}. 

Let us only mention some other known ways to control mean width. 
In some cases, {\em comparison inequalities} for Gaussian processes
can be useful, especially Slepian's and Gordon's; see \cite[Section~3.3]{LT}.
There is also a combinatorial approach to estimating 
the mean width and metric entropy, which is based on 
{\em VC-dimension} and its generalizations; see \cite{Mendelson, RV Annals}.

\subsubsection{Mean width and Gaussian processes}

The theoretical tools of estimating mean width we just mentioned, 
including Sudakov's, Dudley's, Slepian's and Gordon's inequalities, 
have been developed in the context of stochastic processes. 
To see the connection, consider the 
Gaussian random variables $G_{\vu} = \ip{\vg}{\vu}$ indexed by points $\vu \in \R^n$.
The collection of these random variables $(G_{\vu})_{\vu \in K-K}$ forms 
a {\em Gaussian process}, and the mean width measures the size of this process: 
$$
w(K) = \E \sup_{\vu \in K-K} G_{\vu}.
$$
In some sense, any Gaussian process can be approximated by a process
of this form. We will return to the connection between mean width 
and Gaussian processes in Section~\ref{s: M* proof} where we prove the $M^*$ bound.

\subsubsection{Mean width, complexity and effective dimension}	\label{s: eff dim}

In the context of stochastic processes, Gaussian mean width 
(and its non-gaussian variants) 
play an important role in statistical learning theory.
There it is more natural to work with classes $\FF$ of real-valued functions
on $\{1,\ldots,n\}$ than with geometric sets $K \subseteq \R^n$. 
(We identify a vector in $\R^n$ with a function on $\{1,\ldots,n\}$.)
The Gaussian mean width serves as a 
measure of complexity of a function class in statistical learning theory, see \cite{Mendelson2}.
It is sometimes called {\em Gaussian complexity} 
and is usually denoted $\gamma_2(\FF)$.

\smallskip

To get a better feeling of mean width as complexity, 
assume that $K$ lies in the unit Euclidean ball $B_2^n$.
The square of the mean width, $w(K)^2$, 
may be interpreted as the {\em effective dimension} of $K$.
By Example~\ref{ex: algebraic dim}, the effective dimension 
is always bounded by the linear algebraic dimension.   
However, unlike algebraic dimension, the effective dimension is {\em robust} -- 
a small perturbation of $K$ leads to a small change in $w(K)^2$.

\subsection{Random sections of small codimension: $M^*$ bound}		\label{s: M*}

Let us return to the problem we posed in Section~\ref{s: high dim sections?} --
bounding the diameter of random sections $K \cap E$ where $E$ 
is a high-dimensional subspace. 
The following important result in asymptotic convex geometry 
gives a good answer to this question.

\begin{theorem}[$M^*$ bound]					\label{thm: M*}
  Let $K$ be a bounded subset of $\R^n$. Let $E$ be a random 
  subspace of $\R^n$ of a fixed codimension $m$, drawn from the 
  Grassmanian $G_{n, n-m}$ according to the Haar measure. 
  Then
  $$
  \E \diam (K \cap E) \le \frac{C w(K)}{\sqrt{m}}. 
  $$
\end{theorem}

We will prove a stronger version of this result in Section~\ref{s: M* proof}.
The first variant of $M^*$ bound was found by V.~Milman \cite{Milman-Mstar1, Milman-Mstar2};
its present form is due to A.~Pajor and N.~Tomczak-Jaegermann \cite{PT}; 
an alternative argument which yields tight constants was given by Y. Gordon
\cite{Gordon}; an exposition of $M^*$ bound can be found in \cite{Pisier, LT}.

\medskip

To understand the $M^*$ bound better, 
it is helpful to recall from Section~\ref{s: mean width prop}
that $w(K)/\sqrt{n}$ is equivalent to the 
{\em spherical} mean width of $K$. Heuristically,  
the spherical mean width measures the {\em size of the bulk} of $K$. 

For subspaces $E$ of not very high dimension, where $m = \Omega(n)$, 
the $M^*$ bound states that the size of the random section $K \cap E$ 
is bounded by the spherical mean width of $K$. 
In other words, subspaces $E$ of proportional dimension 
{\em passes through the bulk} of $K$ and ignores the outliers (``tentacles''), 
just as Figure~\ref{fig: milman-convex-body-section} illustrates. 
But when the dimension of the subspace $E$ grows toward $n$
(so the codimension $m$ becomes small), the diameter
of $K \cap E$ also grows by a factor of $\sqrt{n/m}$.
This gives a precise control of how $E$ in this case 
{\em interferes with the outliers} of $K$.

\section{From geometry to estimation: linear observations}	
\label{s: estimation linear}

Having completed the excursion into geometry,
we can now return to the high-dimensional estimation problems
that we started to discuss in Section~\ref{s: hde problems}.
To recall, our goal is to estimate an unknown vector
$$
\vx \in K \subseteq \R^n
$$
that lies in a known feasible set $K$, from 
a random observation vector 
$$
\vy = (y_1,\ldots,y_m) \in \R^m,
$$
whose coordinates $y_i$ are random i.i.d. observations 
of $x$. 

So far, we have not been clear about possible distributions of the observations $y_i$.
In this section, we will study perhaps the simplest 
model -- {\em Gaussian linear observations}. 
Consider i.i.d. standard Gaussian vectors
$$
\va_i  \sim N(0,I_n)
$$
and define
$$
y_i = \ip{\va_i}{\vx}, \quad i=1,\ldots,m.
$$
Thus the observation vector $\vy$ depends linearly on $\vx$. This is best expressed 
in a matrix form: 
$$
\vy = A \vx.
$$
Here $A$ in an $m \times n$ Gaussian random matrix, which means that
the entires of $A$ are i.i.d. $N(0,1)$ random variables; the vectors $\va_i$ 
form the rows of $A$.

The interesting regime is when the number of observations is 
smaller than the dimension, i.e. when $m < n$. In this regime, the
problem of estimating $\vx \in \R^n$ from $\vy \in \R^m$ is ill posed. 
(In the complementary regime, where $m \ge n$, the linear system $\vy = A \vx$
is well posed since $A$ has full rank almost surely, so the solution is trivial.)

\subsection{Estimation based on $M^*$ bound}		\label{s: based on M*}

Recall that we know two pieces of information about $\vx$: 
\begin{enumerate}[\qquad 1.]
  \item $\vx$ lies in a known random affine subspace $\{ \vx' :\; A\vx' = \vy \}$;
  \item $\vx$ lies in a known set $K$. 
\end{enumerate}
Therefore, a good estimator of $\vx$ can be obtained by picking any 
vector $\xhat$ from the {\em intersection of these two sets}; see Figure~\ref{fig: estimation-M}. 
Moreover, since just these
two pieces of information about $\vx$ are available, 
such estimator is best possible in some sense.

\begin{figure}[htp]			
  \centering \includegraphics[height=2.1cm]{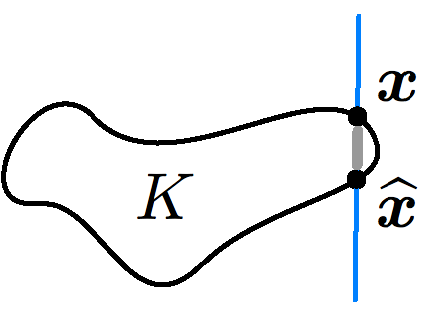} 
  \caption{Estimating $\vx$ by any vector $\xhat$ in the intersection of $K$ with 
    the affine subspace $\{ \vx' :\; A\vx' = \vy \}$}
  \label{fig: estimation-M}	
\end{figure}

How good is such estimate? The maximal error is, of course, the distance
between two farthest points in the intersection of $K$ with the 
affine subspace $\{ \vx' :\; A\vx' = \vy \}$. This distance in turn equals the 
diameter of the section of $K$ by this random subspace. But this diameter
is controlled by $M^*$ bound, Theorem~\ref{thm: M*}. Let us put 
together this argument more rigorously.

\medskip

In the following theorem, the setting is the same as above: $K \subset \R^n$ 
is a bounded subset, $\vx \in K$ is an unknown vector 
and $\vy = A\vx$ is the observation vector, where $A$ is an $m \times n$ Gaussian matrix. 

\begin{theorem}[Estimation from linear observations: feasibility program]	
		\label{thm: estimation feasibility}
  Choose $\xhat$ to be any vector satisfying
  \begin{equation}         \label{eq: feasibility}
  \xhat \in K \quad \text{and} \quad A\xhat = \vy.
  \end{equation}
  Then 
  $$
  \E \sup_{\vx \in K} \|\xhat-\vx\|_2 \le \frac{C w(K)}{\sqrt{m}}. 
  $$
\end{theorem}

\begin{proof}
We apply the $M^*$ bound, Theorem~\ref{thm: M*}, for the set
$K-K$ and the subspace $E = \ker(A)$. Rotation invariance
of Gaussian distribution implies that $E$ is uniformly distributed 
in the Grassmanian $G_{n,n-m}$, as required by the $M^*$ bound.
Moreover, it is straightforward to check that $w(K-K) \le 2 w(K)$.
It follows that
$$
\E \diam ((K-K) \cap E) \le \frac{C w(K)}{\sqrt{m}}. 
$$
It remains to note that since $\xhat, \vx \in K$ and $A\xhat = A\vx = \vy$, 
we have $\xhat-\vx \in (K-K) \cap E$.
\end{proof}

The argument we just described was first suggested by 
S.~Mendelson, A.~Pajor and N.~Tomczak-Jaegermann
\cite{MPT}.

\subsection{Estimation as an optimization problem}		\label{s: optimization}

Let us make one step forward and replace the feasibility program \eqref{eq: feasibility}
by a more flexible {\em optimization} program. 

For this, let us make an additional (but quite mild) assumption that $K$ has non-empty 
interior and is {\em star-shaped}. 
Being star-shaped means that together
with each point, the set $K$ contains the segment joining that point to the origin; 
in other words, 
$$
t K \subseteq K \quad \text{for all } t \in [0,1].
$$
For such set $K$, let us revise the feasibility program \eqref{eq: feasibility}.
Instead of intersecting a fixed set $K$ with the affine subspace $\{ \vx' :\; A\vx' = \vy \}$, 
we may {\em blow up} $K$ (i.e. consider a dilate $tK$ with increasing $t \ge 0$) 
until it touches that subspace. Choose $\xhat$ to be the touching point, 
see Figure~\ref{fig: estimation-optimization}. 
\begin{figure}[htp]			
  \centering \includegraphics[height=2cm]{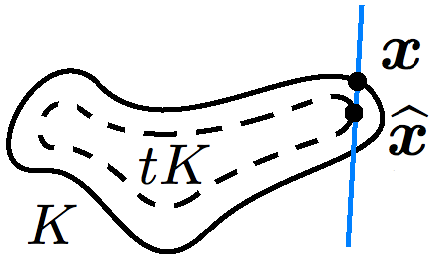} 
  \caption{Estimating $\vx$ by blowing up $K$ until it touches the  
    affine subspace $\{ \vx' :\; A\vx' = \vy \}$}
  \label{fig: estimation-optimization}	
\end{figure}
The fact that $K$ is star-shaped implies that $\xhat$ still belongs to $K$
and (obviously) the affine subspace; thus $\xhat$ satisfies the same error bound 
as in Theorem~\ref{thm: estimation feasibility}.

To express this estimator analytically, it is convenient to use the notion of 
{\em Minkowski functional} of $K$, which associates to each point 
$\vx \in \R^n$ a non-negative number $\|\vx\|_K$ defined by the rule
$$
\|\vx\|_K = \inf \big\{ \l > 0 :\; \l^{-1} \vx \in K \big\}.
$$
Miknowski functionals, also called {\em gauges}, are standard notions in 
geometric functional analysis and convex analysis. Convex analysis textbooks
such as \cite{Rockafellar} offer thorough treatments of this concept.  
We just mention here a couple elementary properties.
First, the function $\vx \mapsto \|\vx\|_K$ is continuous on $\R^n$ and 
it is positive homogeneous (that is, $\|a\vx\|_K = a \|\vx\|_K$ for $a>0$). 
Next, a closed set $K$ is the $1$-sublevel set of its Minkowski functional, 
that is 
$$
K = \{x: \|\vx\|_K \le 1\}.
$$
A typical situation to think of is when $K$ is a symmetric convex body 
(i.e. $K$ is closed, bounded, has non-empty interior and is origin-symmetric);
then $\|\vx\|_K$ defines a {\em norm} on $\R^n$ with $K$ being the unit ball. 

Let us now accurately state an optimization version of Theorem~\ref{thm: estimation feasibility}. 
It is valid for an arbitrary bounded star-shaped set $K$ with non-empty interior.

\begin{theorem}[Estimation from linear observations: optimization program]	
		\label{thm: estimation optimization}
  Choose $\xhat$ to be a solution of the program 
  \begin{equation}         \label{eq: optimization}
  \text{minimize } \|\vx'\|_K \quad \text{subject to} \quad A\vx' = \vy.
  \end{equation}
  Then 
  $$
  \E \sup_{\vx \in K}\|\xhat-\vx\|_2 \le \frac{C w(K)}{\sqrt{m}}. 
  $$
\end{theorem}

\begin{proof}
It suffices to check that $\xhat \in K$; the conclusion would then follow 
from Theorem~\ref{thm: estimation feasibility}.
Both $\xhat$ and $\vx$ satisfy the linear constraint $A\vx' = \vy$. 
Therefore, by choice of $\xhat$, we have
$$
\|\xhat\|_K \le \|\vx\|_K \le 1;
$$
the last inequality is nothing else than our assumption that $\vx \in K$.
Thus $\xhat \in K$ as claimed.
\end{proof}

\subsection{Algorithmic aspects: convex programming} \label{s: algorithmic}

What does it take to solve the optimization problem \eqref{eq: optimization} algorithmically?
If the feasible set $K$ is convex, then \eqref{eq: optimization} is a {\em convex program}. 
In this case, to solve this problem numerically one may tap into an array of available 
convex optimization solvers, in particular interior-point methods \cite{Bental-Nemirovski} and 
proximal-splitting algorithms \cite{Bauschke-Combettes}.

Further, if $K$ is a polytope, then \eqref{eq: optimization} can be cast as 
a {\em linear program}, which widens an array of algorithmic possibilities even further. 
For a quick preview, let us mention that examples of the latter kind will be discussed in detail in Section~\ref{s: sparse recovery},
where we will use $K$ to enforce {\em sparsity}. We will thus choose $K$ to be a ball of $\ell_1$ norm in $\R^n$, 
so the program \eqref{eq: optimization} will minimize $\|\vx'\|_1$ subject to $A\vx' = \vy$. 
This is a typical linear program in the area of compressed sensing.

If $K$ is {\em not} convex, then we can convexify it, thereby replacing $K$ with its convex hull
$\conv(K)$. Convexification does not change the mean width according to 
the remarkable property \eqref{eq: conv wK}.
Therefore, the generally non-convex problem \eqref{eq: optimization} can be relaxed to the 
convex program 
\begin{equation}         \label{eq: estimation convex}
\text{minimize } \|\vx'\|_{\conv(K)} \quad \text{subject to} \quad A\vx' = \vy,
\end{equation}
without compromising the guarantee of estimation stated in 
Theorem~\ref{thm: estimation optimization}. The solution $\xhat$ 
of the convex program \eqref{eq: estimation convex} satisfies
$$
\E \sup_{\vx \in K} \|\xhat-\vx\|_2 
\le \E \sup_{\vx \in \conv(K)} \|\xhat-\vx\|_2
\le \frac{C w(\conv(K))}{\sqrt{m}}
= \frac{C w(K)}{\sqrt{m}}. 
$$

Summarizing, we see that in any case, whether $K$ is convex or not,
the estimation problem reduces to solving an algorithmically tractable convex program. 
Of course, one needs to be able to compute $\|\vz\|_{\conv(K)}$
algorithmically for a given vector $\vz \in \R^n$. This is possible
for many (but not all) feasible sets $K$.

\subsection{Information-theoretic aspects: effective dimension}			\label{s: effective dim}

If we fix a desired error level, for example if we aim for
$$
\E \sup_{\vx \in K} \|\xhat-\vx\|_2 \le 0.01,
$$
then 
$$
m \sim w(K)^2
$$
observations will suffice. The implicit constant factor here is determined 
by the desired error level.

Notice that this result is {\em uniform}. By Markov's inequality, with probability, 
say $0.9$ in $A$ (which determines the observation model)
the estimation is accurate simultaneously for all vectors $\vx \in K$.
Moreover, as we observed in Section~\ref{s: whp},  
the actual probability is much better than $0.9$; it converges to $1$ exponentially fast in 
the number of observations $m$.

The square of the mean width, $w(K)^2$, can be thought of an 
{\em effective dimension} of the feasible set $K$, as we pointed out in 
Section~\ref{s: eff dim}. 

We can summarize our findings as follows. 

\begin{quote}
{\em Using convex programming, one can estimate a vector $\vx$ 
in a general feasible set $K$ from $m$ random linear observations.
A sufficient number of observations $m$ 
is the same as the effective dimension of $K$ (the mean width squared),
up to a constant factor.}
\end{quote}

\section{High dimensional sections: proof of a general $M^*$ bound}		\label{s: M* proof}

Let us give a quick proof of the $M^*$ bound, Theorem~\ref{thm: M*}.
In fact, without much extra work we will be able to derive a more general result
from \cite{PV DCG}. 
First, it would allow us to treat noisy observations of the form $\vy = A\vx + \vnu$. 
Second, it will be generalizable for non-gaussian observations.

\begin{theorem}[General $M^*$ bound]				\label{thm: M* general}
  Let $T$ be a bounded subset of $\R^n$. 
  Let $A$ be an $m \times n$ Gaussian random matrix (with i.i.d. $N(0,1)$ entries).  
  Fix $\e \ge 0$ and consider the set 
  \begin{equation}         \label{eq: Te}
  T_\e := \Big\{ \vu \in T :\; \frac{1}{m} \|A\vu\|_1 \le \e \Big\}.
  \end{equation}
  Then\footnote{The conclusion \eqref{eq: M* general} is stated with the convenetion 
    that $\sup_{\vu \in T_\e} \|u\|_2 = 0$ whenever $T_\e = \emptyset$.}
  \begin{equation}         \label{eq: M* general}
  \E \sup_{\vu \in T_\e} \|\vu\|_2 
  \le \sqrt{\frac{8\pi}{m}} \, \E \sup_{\vu \in T} |\ip{\vg}{\vu}| + \sqrt{\frac{\pi}{2}} \, \e,
  \end{equation}
  where $g \sim N(0,I_n)$ is a standard Gaussian random vector in $\R^n$.
  \end{theorem}

To see that this result contains the classical $M^*$ bound, Theorem~\ref{thm: M*}, 
we can apply it for $T=K-K$, $\e=0$, and identify $\ker(A)$ with $E$. 
In this case, 
$$
T_\e = (K-K) \cap E.
$$ 
It follows that $T_\e \supseteq (K \cap E) - (K \cap E)$, 
so the left hand side in \eqref{eq: M* general} is bounded below by $\diam(K \cap E)$.
The right hand side in \eqref{eq: M* general} by symmetry equals $\sqrt{8\pi/m} \, w(K)$.
Thus, we recover Theorem~\ref{thm: M*} with $C=\sqrt{8\pi}$.

\bigskip

Our proof of Theorem~\ref{thm: M* general} will be based on  
two basic tools in the theory of {\em stochastic processes} --  
symmetrization and contraction.

A stochastic process is simply a collection of random variables $(Z(t))_{t \in T}$
on the same probability space. The index space $T$ can be arbitrary; 
it may be a time interval (such as in Brownian motion) or a subset of $\R^n$
(as will be our case).
To avoid measurability issues, we can assume that $T$ is finite 
by discretizing it if necessary. 

\begin{proposition}					\label{prop: symmetrization contraction}
  Consider a finite collection of stochastic processes $Z_1(t), \ldots, Z_m(t)$
  indexed by $t \in T$. Let $\e_i$ be independent Rademacher random variables
  (that is, $\e_i$ independently take values $-1$ and $1$ with probabilities $1/2$ each).
  Then we have the following. 
  \begin{enumerate}[(i)]
    \item (Symmetrization) 
      $$
      \E \sup_{t \in T} \Big| \sum_{i=1}^m \big[ Z_i(t) - \E Z_i(t) \big] \Big|
      \le 2 \E \sup_{t \in T} \Big| \sum_{i=1}^m \e_i Z_i(t) \Big|.
      $$
    \item (Contraction)
      $$
      \E \sup_{t \in T} \Big| \sum_{i=1}^m \e_i |Z_i(t)| \Big|
      \le 2 \E \sup_{t \in T} \Big| \sum_{i=1}^m \e_i Z_i(t) \Big|.
      $$
    \end{enumerate}
\end{proposition}

Both statements are relatively easy to prove even in greater generality.
For example, taking the absolute values of $Z_i(t)$ in the contraction principle can be 
replaced by applying general Lipschitz functions. Proofs of symmetrization 
and contraction principles can be found in \cite[Lemma~6.3]{LT} and \cite[Theorem~4.12]{LT}, 
respectively.

\subsection{Proof of Theorem~\ref{thm: M* general}}

Let $\va_i^\tran$ denote the rows of $A$; thus $\va_i$ are independent $N(0,I_n)$ random vectors.
The desired bound \eqref{eq: M* general} would follow 
from the deviation inequality
\begin{equation}         \label{eq: M* deviation}
\E \sup_{\vu \in T} \Big| \frac{1}{m} \sum_{i=1}^m |\ip{\va_i}{\vu}| - \sqrt{\frac{2}{\pi}} \, \|\vu\|_2 \Big|
\le \frac{4}{\sqrt{m}} \, \E \sup_{\vu \in T} |\ip{\vg}{\vu}|.
\end{equation}
Indeed, if this inequality holds, then same is true if we replace $T$ by 
the smaller set $T_\e$ in the left hand side of \eqref{eq: M* deviation}. 
But for $\vu \in T_\e$, we have 
$\frac{1}{m} \sum_{i=1}^m |\ip{\va_i}{\vu}| = \frac{1}{m}\|A\vu\|_1 \le \e$, 
and the bound \eqref{eq: M* general} follows by triangle inequality.

The rotation invariance of Gaussian distribution 
implies that 
\begin{equation}         \label{eq: expectation ip}
\E |\ip{\va_i}{\vu}| = \sqrt{\frac{2}{\pi}} \, \|\vu\|_2.
\end{equation}
Thus, using symmetrization and then contraction inequalities
from Proposition~\ref{prop: symmetrization contraction}, we can 
bound the left side of \eqref{eq: M* deviation} by
\begin{equation}         \label{eq: sym contr applied}
4 \E \sup_{\vu \in T} \Big| \frac{1}{m} \sum_{i=1}^m \e_i \ip{\va_i}{\vu} \Big|
= 4 \E \sup_{\vu \in T} \left| \ip{\frac{1}{m} \sum_{i=1}^m \e_i \va_i}{\vu} \right|.
\end{equation}
Here $\e_i$ are independent Rademacher variables.

Conditioning on $\e_i$ and using rotation invariance, we see that 
the random vector 
$$
\vg := \frac{1}{\sqrt{m}} \sum_{i=1}^m \e_i \va_i
$$
has distribution $N(0,I_n)$. Thus \eqref{eq: sym contr applied} can be written as
$$
\frac{4}{\sqrt{m}} \E \sup_{\vu \in T} |\ip{\vg}{\vu}|.
$$
This proves \eqref{eq: M* deviation} and completes the proof 
of Theorem~\ref{thm: M* general}. \qed

\subsection{From expectation to overwhelming probability}			\label{s: whp}

The $M^*$ bound that we just proved, and in fact all results in this survey, 
are stated in terms of expected value for simplicity of presentation. 
One can upgrade them to estimates with overwhelming probability using {\em concentration of measure},
see \cite{Ledoux}.
We will mention how do this for the results we just proved; the reader should be able 
to do so for further results as well.

Let us first obtain a high-probability version of the deviation inequality \eqref{eq: M* deviation}
using the {\em Gaussian concentration inequality}.
We will consider the deviation 
$$
Z(A) := \sup_{\vu \in T} \Big| \frac{1}{m} \sum_{i=1}^m |\ip{\va_i}{\vu}| - \sqrt{\frac{2}{\pi}} \, \|\vu\|_2 \Big|
$$
as a function of the matrix $A \in \R^{m \times n}$. Let us show that it is a Lipschitz function 
on $\R^{m \times n}$ equipped with Frobenius norm $\|\cdot\|_F$ (which is the same as the 
Euclidean norm on $\R^{mn}$).  
Indeed, two applications of the triangle inequality followed by two applications of the Cauchy-Schwarz 
inequality imply that for matrices $A$ and $B$ with rows $\va_i^\tran$ and $\vb_i^\tran$ respectively, 
we have
\begin{align*}
|Z(A) - Z(B)| 
&\leq \sup_{\vu \in T} \frac{1}{m} \sum_{i=1}^m |\ip{\va_i - \vb_i}{\vu} | \\
&\leq \frac{d(T)}{m} \sum_{i=1}^m \|\va_i - \vb_i\|_2 \quad (\text{where } d(T) = \max_{\vu \in T} \|\vu\|_2) \\
&\leq \frac{d(T)}{\sqrt{m}} \|A - B\|_F.
\end{align*}
Thus the function $A \mapsto Z(A)$ has Lipschitz constant bounded by $d(K)/\sqrt{m}$.  
We may now bound the deviation probability for $Z$ using the Gaussian concentration inequality (see \cite[Equation 1.6]{LT}) as follows:
$$
\Pr{|Z - \E Z| \geq t} \leq 2 \exp \Big(-\frac{mt^2}{2d(T)^2}\Big), \quad t \ge 0.
$$
This is a high-probability version of the deviation inequality \eqref{eq: M* deviation}.

Using this inequality, one quickly deduces a corresponding high-probability version of 
Theorem~\ref{thm: M* general}. It states that
$$
\sup_{\vu \in T_\e} \|\vu\|_2 
  \le \sqrt{\frac{8\pi}{m}} \, \E \sup_{\vu \in T} |\ip{\vg}{\vu}| + \sqrt{\frac{\pi}{2}} \, (\e+t)
$$
with probability at least $1-2\exp(-mt^2 / 2d(T)^2)$.

As before, we obtain from this the following {\em high-probability version of the $M^*$
bound}, Theorem~\ref{thm: M* general}. It states that 
$$
\diam(K \cap E) \le \frac{C w(K)}{\sqrt{m}} + Ct
$$
with probability at least $1-2\exp(-mt^2 / 2\diam(K)^2)$.

\section{Consequences: estimation from noisy linear observations}		\label{s: noisy}

Let us apply the general $M^*$ bound, Theorem~\ref{thm: M* general}, 
to estimation problems. This will be even more straightforward
 than our application of the standard $M^*$ bound in
Section~\ref{s: estimation linear}. Moreover, we will now be able to 
treat {\em noisy} observations.

Like before, our goal is to estimate an unknown vector $\vx$ that lies in a known 
feasible set $K \subset \R^n$, from a random observation vector $\vy \in \R^m$.
This time we assume that, for some known level of noise $\e \ge 0$, we have 
\begin{equation}         \label{eq: noisy observations}
\vy = A\vx + \vnu, \qquad \frac{1}{m} \|\vnu\|_1 
= \frac{1}{m} \sum_{i=1}^m |\nu_i| \le \e.
\end{equation}
Here $A$ is an $m \times n$ Gaussian matrix as before. 
The noise vector $\vnu$ may be unknown and have arbitrary structure. 
In particular $\vnu$ may depend on $A$, so even adversarial errors are allowed.
The $\ell_1$ constraint in \eqref{eq: noisy observations} can clearly be replaced by 
the stronger $\ell_2$ constraint 
$$
\frac{1}{m} \|\vnu\|_2^2 
= \frac{1}{m} \sum_{i=1}^m \nu_i^2 
\le \e^2.
$$

The following result is a generalization of Theorem~\ref{thm: estimation feasibility} 
for noisy observations  \eqref{eq: noisy observations}. As before, it is valid for any bounded 
set $K \subset \R^n$.

\begin{theorem}[Estimation from noisy linear observations: feasibility program]	
		\label{thm: estimation noisy feasibility}
  Choose $\xhat$ to be any vector satisfying
  \begin{equation}         \label{eq: noisy feasibility}
  \xhat \in K \quad \text{and} \quad \frac{1}{m} \|A\xhat - \vy\|_1 \le \e.
  \end{equation}
  Then 
  $$
  \E \sup_{\vx \in K} \|\xhat-\vx\|_2 \le \sqrt{8\pi} \, \left( \frac{w(K)}{\sqrt{m}} + \e \right).
  $$
\end{theorem}

\begin{proof}
We apply the general $M^*$ bound, Theorem~\ref{thm: M* general}, for the set $T = K-K$, 
and with $2\e$ instead of $\e$. It follows that
$$
\E \sup_{u \in T_{2\e}} \|\vu\|_2
\le \sqrt{\frac{8\pi}{m}} \, \E \sup_{\vu \in T} |\ip{\vg}{\vu}| + \sqrt{2\pi} \, \e
\le \sqrt{8\pi} \, \left( \frac{w(K)}{\sqrt{m}} + \e \right).
$$
The last inequality follows from the definition of mean width and the symmetry of $T$.

To finish the proof, it remains to check that 
\begin{equation}         \label{eq: T2e}
\xhat-\vx \in T_{2\e}.
\end{equation}
To prove this, first note that $\xhat, \vx \in K$, so 
$\xhat-\vx \in K-K = T$.
Next, by triangle inequality, we have 
$$
\frac{1}{m} \|A(\xhat - \vx)\|_1
= \frac{1}{m} \|A\xhat - \vy + \vnu\|_1
\le \frac{1}{m} \|A\xhat - \vy\|_1 + \frac{1}{m} \|\vnu\|_1
\le 2 \e.
$$
The last inequality follows from \eqref{eq: noisy observations} and \eqref{eq: noisy feasibility}.
We showed that the vector $\vu = \xhat-\vx$ satisfies both constraints that define 
$T_{2\e}$ in \eqref{eq: Te}. Hence \eqref{eq: T2e} holds, and the proof 
of the theorem is complete.
\end{proof}

And similarly to Theorem~\ref{thm: estimation optimization}, we can 
cast estimation as an optimization (rather than feasibility) program. 
As before, it is valid for any bounded star-shaped set $K \subset \R^n$
with nonempty interior. 

\begin{theorem}[Estimation from noisy linear observations: optimization program]	
		\label{thm: estimation noisy optimization}
  Choose $\xhat$ to be a solution to the program 
  \begin{equation}         \label{eq: estimation noisy optimization}
  \text{minimize } \|\vx'\|_K \quad \text{subject to} \quad \frac{1}{m} \|A\vx' - \vy\|_1 \le \e.
  \end{equation}
  Then 
  $$
  \E \sup_{\vx \in K} \|\xhat-\vx\|_2 \le \sqrt{8\pi} \, \left( \frac{w(K)}{\sqrt{m}} + \e \right).
  $$
\end{theorem}

\begin{proof}
It suffices to check that $\xhat \in K$; the conclusion would then follow 
from Theorem~\ref{thm: estimation noisy feasibility}.
Note first that by choice of $\xhat$ we have 
$\frac{1}{m} \|A\xhat - \vy\|_1 \le \e$,
and by assumption \eqref{eq: noisy observations} we have
$\frac{1}{m} \|A\vx - \vy\|_1 = \frac{1}{m} \|\vnu\|_1 \le \e$.
Thus both $\xhat$ and $\vx$ satisfy the constraint in \eqref{eq: estimation noisy optimization}.
Therefore, by choice of $\xhat$, we have
$$
\|\xhat\|_K \le \|\vx\|_K \le 1;
$$
the last inequality is nothing else than our assumption that $\vx \in K$.
It follows $\xhat \in K$ as claimed. 
\end{proof}

The remarks about algorithmic aspects of estimation made in 
Sections~\ref{s: algorithmic} and \ref{s: effective dim}
apply also to the results of this section. In particular, the
estimation from noisy linear observations \eqref{eq: noisy observations} 
can be formulated as a {\em convex program}.

\section{Applications to sparse recovery and regression}				\label{s: sparse recovery}

Remarkable examples of feasible sets $K$ with low complexity
come from the notion of {\em sparsity}. 
Consider the set $K$ of all unit $s$-sparse vectors in $\R^n$. 
As we mentioned in Example~\ref{ex: sparse mean width}, the mean 
width of $K$ is 
$$
w(K) \sim s \log(n/s).
$$
According to the interpretation we discussed in Section~\ref{s: effective dim}, 
this means that the {\em effective dimension} of $K$ is of order $s \log(n/s)$.
Therefore,
$$
m \sim s \log(n/s)
$$
observations should suffice to estimate any $s$-sparse vector in $\R^n$.
Results of this type form the core of {\em compressed sensing}, 
a young area of signal processing, see \cite{DDEK, Kutyniok, CGLP, FR}.

\medskip

In this section we consider a more general model, where an unknown vector $\vx$
has a sparse representation {\em in some dictionary}.

We will specialize Theorem~\ref{thm: estimation noisy optimization} 
to the sparse recovery problem. 
The convex program will in this case amount to minimizing the $\ell_1$
norm of the coefficients. 
We will note that the notion of sparsity can be relaxed to accommodate
approximate, or ``effective'', sparsity. Finally, 
we will observe that the estimate $\xhat$ is most often unique and $m$-sparse.

\subsection{Sparse recovery for general dictionaries}

Let us fix a {\em dictionary} of vectors $\vd_1,\ldots,\vd_N \in \R^n$, 
which may be arbitrary (even linearly dependent). 
The choice of a dictionary depends on the application; common examples include 
unions of orthogonal bases and more generally tight frames (in particular, Gabor frames).
See \cite{CDS, DE, DNW, RSW} for an introduction to 
sparse recovery problems with general dictionaries. 

Suppose an unknown vector $\vx \in \R^n$ is {\em $s$-sparse
in the dictionary $\{\vd_i\}$}. 
This means that $\vx$ can be represented
as a linear combination of at most $s$ dictionary elements, i.~e. 
\begin{equation}         \label{eq: sparse representation}
\vx = \sum_{i=1}^N \a_i \vd_i \quad \text{with at most $s$ non-zero coefficients $\a_i \in \R$}.
\end{equation}

As in Section~\ref{s: noisy}, our goal is to recover $\vx$
from a noisy observation vector $\vy \in \R^m$ of the form
$$
\vy = A\vx + \vnu, \qquad \frac{1}{m} \|\vnu\|_1 
= \frac{1}{m} \sum_{i=1}^m |\nu_i| \le \e.
$$
Recall that $A$ is a known $m \times n$ Gaussian matrix, and
and $\vnu$ is an unknown noise vector, which can have arbitrary structure
(in particular, correlated with $A$). 

Theorem~\ref{thm: estimation noisy optimization} will quickly imply 
the following recovery result. 

\begin{theorem}[Sparse recovery: general dictionaries]			\label{thm: sparse recovery}
  Assume for normalization that all dictionary vectors satisfy $\|\vd_i\|_2 \le 1$. 
  Choose $\xhat$ to be a solution to the convex program 
  \begin{equation}         \label{eq: sparse recovery program}
  \text{minimize } \|\valpha'\|_1 \text{ such that } 
  \vx' = \sum_{i=1}^N \a_i' \vd_i 
  \text{ satisfies } \frac{1}{m} \|A\vx' - \vy\|_1 \le \e.
  \end{equation}
  Then 
  $$
  \E \|\xhat-\vx\|_2 
  \le C \sqrt{\frac{s \log N}{m}} \cdot \|\valpha\|_2 + \sqrt{2\pi} \; \e.
  $$
\end{theorem}

\begin{proof}
Consider the sets 
$$
\bar{K} := \conv\{ \pm \vd_i\}_{i=1}^N, \quad K : = \|\valpha\|_1 \cdot \bar{K}.
$$
Representation \eqref{eq: sparse representation} implies that $\vx \in K$, 
so it makes sense to apply 
Theorem~\ref{thm: estimation noisy optimization} for $K$.

Let us first argue that the optimization program in Theorem~\ref{thm: estimation noisy optimization}
can be written in the form \eqref{eq: sparse recovery program}.
Observe that we can replace $\|\vx'\|_K$ by $\|\vx'\|_{\bar{K}}$ 
in the optimization problem \eqref{eq: estimation noisy optimization} without changing its solution.
(This is because $\|\vx'\|_{\bar{K}} = \|\valpha\|_1 \cdot \|\vx'\|_K$ and $\|\valpha\|_1$
is a constant value.)
Now, by definition of $\bar{K}$, we have
$$
\|\vx'\|_{\bar{K}} = \min \Big\{ \|\valpha'\|_1 :\; \vx' = \sum_{i=1}^N \a_i' \vd_i \Big\}.
$$
Therefore, the optimization programs \eqref{eq: estimation noisy optimization}
and \eqref{eq: sparse recovery program} are indeed equivalent.

Next, to evaluate the error bound in Theorem~\ref{thm: estimation noisy optimization}, 
we need to bound the mean width of $K$. The convexification property \eqref{eq: conv wK} and 
Example~\ref{ex: finite set mean width} yield 
$$
w(K) = \|\valpha\|_1 \cdot w(\bar{K}) \le C \|\valpha\|_1 \cdot \sqrt{\log N}.
$$
Putting this into the conclusion of Theorem~\ref{thm: estimation noisy optimization},
we obtain the error bound 
$$
\E \sup_{\vx \in K} \|\xhat-\vx\|_2 
\le \sqrt{8\pi} \; C \; \sqrt{\frac{\log N}{m}} \cdot \|\valpha\|_1 + \sqrt{2\pi} \; \e.
$$
To complete the proof, it remains to note that 
\begin{equation}         \label{eq: ell1 ell2}
\|\valpha\|_1 \le \sqrt{s} \cdot \|\valpha\|_2,
\end{equation}
since $\valpha$ is $s$-sparse, i.e. it has only $s$ non-zero coordinates.
\end{proof}

\subsection{Remarkable properties of sparse recovery} 

Let us pause to look more closely at the statement of Theorem~\ref{thm: sparse recovery}.

\subsubsection{General dictionaries}
  Theorem~\ref{thm: sparse recovery} is very flexible with respect to the 
  choice of a dictionary $\{\vd_i\}$. 
  Note that there are essentially {\em no restrictions on the dictionary}. 
  (The normalization assumption $\|\vd_i\|_2 \le 1$ can be dispensed 
  of at the cost of increasing the error bound by the factor 
  of $\max_i \|\vd_i\|_2$.) In particular, the dictionary may be {\em linearly dependent}.

\subsubsection{Effective sparsity}				\label{s: effective sparsity}
  The reader may have noticed that the proof of Theorem~\ref{thm: sparse recovery}
  used sparsity in a quite mild way, only through inequality \eqref{eq: ell1 ell2}.
  So the result is still true for vectors $\vx$ that are {\em approximately sparse}
  in the dictionary.  
  Namely, the Theorem~\ref{thm: sparse recovery} will hold if we replace the exact notion of 
  sparsity (the number of nonzero coefficients) 
  by the more flexible notion of {\em effective sparsity}, defined as  
  $$
  \text{effective sparsity} (\valpha) := (\|\valpha\|_1 / \|\valpha\|_2)^2.
  $$
  It is now clear how to extend sparsity in a dictionary \eqref{eq: sparse representation}
  to approximate sparsity. We can say that a vector $\vx$ is {\em effectively
  $s$-sparse in a dictionary $\{d_i\}$} if it can be represented as 
  $\vx = \sum_{i=1}^N \a_i \vd_i$ where the coefficient vector 
  $\va = (\a_1,\ldots,\a_N)$ is effectively $s$-sparse.
  
  The effective sparsity is clearly bounded by the exact sparsity, and it is 
  robust with respect to small perturbations.

\subsubsection{Linear programming}
  The convex programs \eqref{eq: sparse recovery program} and \eqref{eq: sparse recovery}
  can be reformulated as linear programs.
   This can be done by introducing new variables $u_1,\ldots, u_N$; 
  instead of minimizing $\|\valpha'\|_1$ in \eqref{eq: sparse recovery program}, we can 
  equivalently minimize the linear function $\sum_{i=1}^N u_i$ subject to the additional 
  linear constraints $-u_i \le \alpha'_i \le u_i$, $i=1,\ldots,N$. 
  In a similar fashion, one can replace the convex constraint $\frac{1}{m} \|A\vx' - \vy\|_1 \le \e$ 
  in \eqref{eq: sparse recovery program} by $n$ linear constraints.
    
\subsubsection{Estimating the coefficients of sparse representation}
It is worthwhile to notice that as a result of solving the convex recovery
program \eqref{eq: sparse recovery program}, we obtain not only an estimate $\xhat$
of the vector $\vx$, but also an estimate $\widehat{\valpha}$ of the 
{\em coefficient vector} in the representation $\vx = \sum \a_i \vd_i$.

\subsubsection{Sparsity of solution}
The solution of the sparse recovery problem \eqref{eq: sparse recovery program}
may not be exact in general, that is $\xhat \ne \vx$ can happen. 
This can be due to several factors -- the generality of the dictionary, approximate 
(rather than exact) sparsity of $\vx$ in the dictionary, and the noise $\nu$
in the observations.  
But even in this general situation, {\em the solution $\vx$ is still $m$-sparse},
in all but degenerate cases. We will now state and prove this known fact 
(see \cite{FR}).

\begin{proposition}[Sparsity of solution]				\label{prop: sparse solution}
  Assume that a given convex recovery program \eqref{eq: sparse recovery program}
  has a unique solution $\widehat{\valpha}$ for the coefficient vector. 
  Then $\widehat{\valpha}$ is $m$-sparse, and consequently  
  $\xhat$ is $m$-sparse in the dictionary $\{\vd_i\}$.
  This is true even in presence of noise in observations, and even when 
  no sparsity assumptions on $\vx$ are in place.
\end{proposition}

\begin{proof}
The result follows by simple dimension considerations. 
  First note that the constraint on $\valpha'$ in the optimization problem 
  \eqref{eq: sparse recovery program} can be written in the form 
  \begin{equation}         \label{eq: ell1 constraint}
  \frac{1}{m}\|AD\valpha' - \vy\|_1 \le \e,
  \end{equation}
  where $D$ is the $n \times N$ matrix whose columns are the dictionary vectors $\vd_i$.
  Since matrix $AD$ has dimensions $m \times N$, the 
  constraint defines a cylinder in $\R^N$ 
  whose infinite directions are formed by the kernel of $AD$, which has dimension 
  at least $N-m$. Moreover, this cylinder is a polyhedral set (due to the $\ell_1$ norm defining it), 
  so it has no faces of dimension smaller than $N-m$.   
  
  On the other hand, the level sets of the objective function $\|\valpha'\|_1$ are also 
  polyhedral sets; they are dilates of the unit $\ell_1$ ball. 
  The solution $\widehat{\valpha}$ of the optimization problem 
  \eqref{eq: sparse recovery program} is thus a point in $\R^N$ where 
  the smallest dilate of the $\ell_1$ ball touches the cylinder.  
  The uniqueness of solution means that a touching point is unique.
  This is illustrated in Figure~\ref{fig: touching faces}. 
  
  \begin{figure}[htp]			
  \centering \includegraphics[height=3cm]{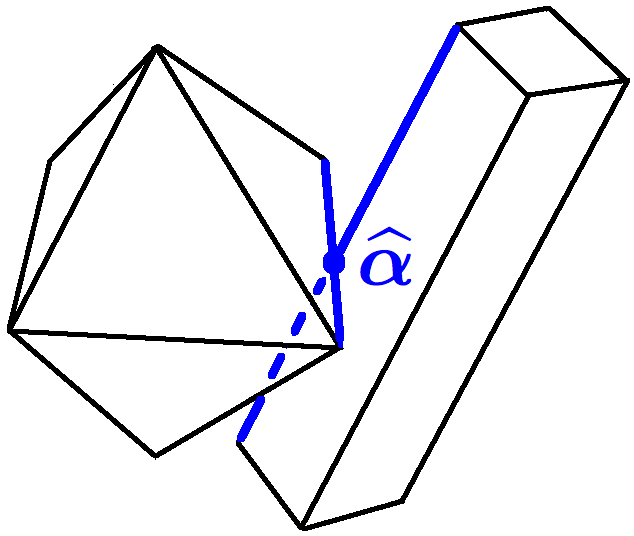} 
  \caption{Illustration for the proof of Proposition~\ref{prop: sparse solution}.
    The polytope on the left represents a level set of the $\ell_1$ ball. 
    The cylinder on the right represents the vectors $\valpha'$ satisfying the 
    constraint \eqref{eq: ell1 constraint}. The two polyhedral sets touch at point $\widehat{\valpha}$.}
  \label{fig: touching faces}	
  \end{figure}
  
  Consider the faces of these two polyhedral sets of smallest dimensions 
  that contain the touching point; we may call these the touching faces. 
  The touching face of the cylinder has dimension at least $N-m$, as all of its faces do. 
  Then the touching face of the $\ell_1$ ball must have dimension at most $m$, otherwise
  the two touching faces would intersect by more than one point. This
  translates into the $m$-sparsity of the solution $\widehat{\valpha}$, as claimed.
\end{proof}

In view of Proposition~\ref{prop: sparse solution}, we can ask when the 
solution $\widehat{\valpha}$ of the convex program \eqref{eq: sparse recovery program} 
is unique. This does not always happen; for example this fails if $\vd_1=\vd_2$. 
  
Uniqueness of solutions of optimization problems like \eqref{eq: sparse recovery program} 
is extensively studied \cite{FR}. Let us mention here a cheap way to obtain uniqueness.
This can be achieved by an arbitrarily small generic perturbation of the dictionary elements, 
such as adding a small independent Gaussian vector to each $\vd_i$.
Then one can see that the solution $\widehat{\valpha}$ (and therefore $\xhat$ as well)
are unique almost surely. Invoking Proposition~\ref{prop: sparse solution} we see 
that $\xhat$ is $m$-sparse in the perturbed dictionary.

\subsection{Sparse recovery for the canonical dictionary}

Let us illustrate Theorem~\ref{thm: sparse recovery} for the simplest example 
of a dictionary -- the canonical basis of $\R^n$:
$$
\{\vd_i\}_{i=1}^n = \{\ve_i\}_{i=1}^n.
$$
In this case, our assumption is that an unknown vector $\vx \in \R^n$ is {\em $s$-sparse}
in the usual sense, meaning that $\vx$ has at most $s$ non-zero coordinates, 
or {\em effectively $s$-sparse} as in Section~\ref{s: effective sparsity}. 
Theorem~\ref{thm: sparse recovery} then reads as follows.

\begin{corollary}[Sparse recovery]		\label{cor: sparse recovery}
  Choose $\xhat$ to be a solution to the convex program 
  \begin{equation}         \label{eq: sparse recovery}
  \text{minimize } \|\vx'\|_1 \text{ subject to } 
  \frac{1}{m} \|A\vx' - \vy\|_1 \le \e.
  \end{equation}
  Then 
  $$
  \E \|\xhat-\vx\|_2 
  \le C \sqrt{\frac{s \log n}{m}} \cdot \|\vx\|_2 + \sqrt{2\pi} \; \e.	\qquad \qed
  $$

\end{corollary}

Sparse recovery results like Corollary~\ref{cor: sparse recovery} form the 
core of the area of {\em compressed sensing}, see \cite{DDEK, Kutyniok, CGLP, FR}.

In the noiseless case ($\e=0$) and for sparse (rather then effectively sparse)
vectors, one may even hope to recover $\vx$ {\em exactly}, meaning that
$\xhat = \vx$ with high probability. Conditions for exact recovery are now well 
understood in compressed sensing. We will discuss some exact 
recovery problems in Section~\ref{s: exact}.

\medskip

We can summarize Theorem~\ref{thm: sparse recovery} 
and the discussion around it as follows.

\begin{quote}
{\em Using linear programming, one can approximately recover a vector $\vx$ that
is $s$-sparse (or effectively $s$-sparse) in a general dictionary of size $N$, from
$m \sim s \log N$
random linear observations.}
\end{quote}

\subsection{Application: linear regression with constraints}			\label{s: regression}
The noisy estimation problem \eqref{eq: noisy observations} is equivalent
to {\em linear regression} with constraints. So in this section we will translate  
the story into the statistical language. We present here just one class of examples 
out of a wide array of statistical problems; we refer the reader to \cite{BG, Wainwright} 
for a recent review of high dimensional estimation problems from a statistical viewpoint. 

Linear regression is a model of linear relationship between 
one dependent variable and $n$ explanatory variables. It is usually written as 
$$
\vy = X \vbeta + \vnu.
$$
Here $X$ is an $n \times p$ matrix which contains a sample 
of $n$ observations of $p$ explanatory variables; 
$\vy \in \R^n$ represents a sample of $n$ observations of the dependent variable; 
$\vbeta \in \R^p$ is a coefficient vector;
$\vnu \in \R^n$ is a noise vector. 
We assume that $X$ and $\vy$ are known, while $\vbeta$ and $\vnu$ are unknown. 
Our goal is to estimate $\vbeta$.

We discussed a classical formulation of linear regression. In addition, we
often know, believe, or want to enforce some properties about the coefficient vector $\vbeta$, 
(for example, sparsity). We can express such extra information
as the assumption that 
$$
\vbeta \in K
$$
where $K \subset \R^p$ is a known feasible set. 
Such problem may be called a {\em linear regression with constraints}.
 
The high dimensional estimation results we have seen so far can be translated into 
the language of regression in a straightforward way. 
Let us do this for
Theorem~\ref{thm: estimation noisy optimization}; the interested reader
can make a similar translation or other results.

We assume that the explanatory variables are independent $N(0,1)$, 
so the matrix $X$ has all i.i.d. $N(0,1)$ entries. This requirement may be too strong
in practice; however see Section~\ref{s: sub-gaussian} 
on relaxing this assumption. The noise vector $\vnu$ is allowed have arbitrary structure
(in particular, it can be correlated with $X$). We assume that its magnitude is controlled:  
$$
\frac{1}{n} \|\vnu\|_1 = \frac{1}{n} \sum_{i=1}^n |\nu_i| \le \e
$$
for some known noise level $\e$.
Then we can restate Theorem~\ref{thm: estimation noisy optimization} in the 
following way.

\begin{theorem}[Linear regression with constraints]	
		\label{thm: regression with constraints}
  Choose $\vbetahat$ to be a solution to the program 
  $$  
  \text{minimize } \|\vbeta'\|_K \quad \text{subject to} \quad \frac{1}{n} \|X\vbeta' - \vy\|_1 \le \e.
  $$
  Then 
  $$
  \E \sup_{\vbeta \in K} \|\vbetahat-\vbeta\|_2 \le \sqrt{8\pi} \, \left( \frac{w(K)}{\sqrt{n}} + \e \right).
  \qquad \quad \qed
  $$
\end{theorem}

\section{Extensions from Gaussian to sub-gaussian distributions}		\label{s: sub-gaussian}

So far, all our results were stated for Gaussian distributions.
Let us show how to relax this assumption. In this section, we
will modify the proof of the $M^*$ bound, Theorem~\ref{thm: M* general} 
for general {\em sub-gaussian} distributions, and indicate the 
consequences for the estimation problem.
A result of this type was proved in \cite{MPT} with a much more complex argument.

\subsection{Sub-gaussian random variables and random vectors}
A systematic introduction into sub-gaussian distributions can be found
in Sections 5.2.3 and 5.2.5 of \cite{V tutorial}; here we briefly mention the 
basic definitions.
According to one of the several equivalent definitions, 
a random variable $X$ is {\em sub-gaussian} if
$$
\E \exp(X^2/\psi^2) \le e.
$$
for some $\psi>0$. The smallest $\psi$ is called the {\em sub-gaussian norm}
and is denoted $\|X\|_\psitwo$. Normal and all bounded random variables
are sub-gaussian, while exponential random variables are not. 

The notion of sub-gaussian distribution transfers to higher dimensions as follows. 
A random vector $\vX \in \R^n$ is called sub-gaussian if all 
one-dimensional marginals $\ip{\vX}{\vu}$, $\vu \in \R^n$, 
are sub-gaussian random variables. The sub-gaussian norm of $\vX$ 
is defined as
\begin{equation}         \label{eq: sub-gaussian vector}
\|\vX\|_\psitwo := \sup_{\vu \in S^{n-1}} \|\ip{\vX}{\vu}\|_\psitwo
\end{equation}
where, as before, $S^{n-1}$ denotes the Euclidean sphere in $\R^n$.
Recall also that the random vector $\vX$ is called {\em isotropic} if
$$
\E \vX \vX^\tran = I_n.
$$
Isotropy is a scaling condition; any distribution in $\R^n$ which is not supported
in a low-dimensional subspace can be made isotropic by an appropriate 
linear transformation. To illustrate this notion with a couple of quick examples, 
one can check that $N(0,I_n)$ and the uniform distribution on the discrete cube $\{-1,1\}^n$
are isotropic and sub-gaussian distributions.

\subsection{$M^*$ bound for sub-gaussian distributions}

Now we state and prove a version of $M^*$ bound, Theorem~\ref{thm: M* general},
for general {\em sub-gaussian} distributions. It is a variant of a result from \cite{MPT}.

\begin{theorem}[General $M^*$ bound for sub-gaussian distributions]	 \label{thm: M* sub-gaussian}
  Let $T$ be a bounded subset of $\R^n$. 
  Let $A$ be an $m \times n$ matrix whose rows $\va_i$ are i.i.d., mean zero, 
  isotropic and sub-gaussian random vectors in $\R^n$. Choose $\psi \ge 1$ so that 
  \begin{equation}         \label{eq: ai sub-gaussian}
  \|\va_i\|_\psitwo \le \psi, \quad i=1,\ldots,m.
  \end{equation}
  Fix $\e \ge 0$ and consider the set 
  $$
  T_\e := \Big\{ \vu \in T :\; \frac{1}{m} \|A\vu\|_1 \le \e \Big\}.
  $$
  Then 
  $$
  \E \sup_{\vu \in T_\e} \|\vu\|_2 
  \le C \psi^4 \Big( \frac{1}{\sqrt{m}} \, \E \sup_{\vu \in T} |\ip{\vg}{\vu}| 
    + \e \Big),
  $$
  where $g \sim N(0,I_n)$ is a standard Gaussian random vector in $\R^n$.
\end{theorem}

A proof of this result is an extension of the proof of the Gaussian $M^*$ bound, 
Theorem~\ref{thm: M* general}.
Most of that argument generalizes to sub-gaussian distributions in a standard way.
The only non-trivial new step will be based on the deep {\em comparison theorem 
for sub-gaussian processes} due to X.~Fernique and M.~Talagrand, 
see \cite[Section~2.1]{Talagrand}.
Informally, the result states that any sub-gaussian process is dominated by
a Gaussian process with the same (or larger) increments. 

\begin{theorem}[Fernique-Talagrand's comparison theorem]		\label{thm: fernique-talagrand}
  Let $T$ be an arbitrary set.\footnote{We can assume $T$ to be finite to avoid
    measurability complications, and then proceed by approximation; see e.g. \cite[Section~2.2]{LT}.} 
  Consider a Gaussian random process $(G(\vt))_{\vt \in T}$ and a sub-gaussian
  random process $(H(\vt))_{\vt \in T}$. 
  Assume that $\E G(\vt) = \E H(\vt) = 0$ for all $\vt \in T$. 
  Assume also that for some $M>0$,
  the following increment comparison holds:\footnote{The increment 
    comparison may look better if we replace the $L_2$ norm in the right hand 
    side by $\psitwo$ norm. Indeed, it is easy to see that 
    $\|G(\vs) - G(\vt)\|_\psitwo \asymp (\E \|G(\vs) - G(\vt)\|_2^2)^{1/2}$.}
  $$
  \|H(\vs) - H(\vt)\|_\psitwo \le M \, (\E \|G(\vs) - G(\vt)\|_2^2)^{1/2} \quad \text{for all } \vs, \vt \in T.
  $$
  Then 
  $$
  \E \sup_{\vt \in T} H(\vt) \le C M \, \E \sup_{\vt \in T} G(\vt).
  $$
\end{theorem}

This theorem is a combination of a result of X.~Fernique~\cite{Fernique} that bounds 
$\E \sup_{\vt \in T} H(\vt)$ above by the so-called {\em majorizing measure} of $T$,
and a result of M.~Talagrand~\cite{Talagrand-maj} that bounds $\E \sup_{\vt \in T} G(\vt)$
below by the same majorizing measure of $T$.

\begin{proof}[Proof of Theorem~\ref{thm: M* sub-gaussian}]
Let us examine the proof of the Gaussian $M^*$ bound, Theorem~\ref{thm: M* general},
check where we used Gaussian assumptions, and try to accommodate
sub-gaussian assumptions instead.

\medskip

The first such place is identity \eqref{eq: expectation ip}. We claim that a version 
of it still holds for the sub-gaussian random vector $\va$, namely
\begin{equation}         \label{eq: expectation ip sub-gaussian}
\|\vu\|_2 \le C_0 \psi^3 \, \E_{\va} |\ip{\va}{\vu}|
\end{equation}
where $C_0$ is an absolute constant.\footnote{We should mention that a reverse 
  inequality also holds: by isotropy, one has 
  $\E_{\va} |\ip{\va}{\vu}| \le (\E_{\va} \ip{\va}{\vu}^2)^{1/2} = \|\vu\|_2$.
  However, this inequality will not be used in the proof.}

To check \eqref{eq: expectation ip sub-gaussian}, we can assume that 
$\|\vu\|_2 = 1$ by dividing both sides by $\|\vu\|_2$ if necessary.
Then $Z := \ip{\va}{\vu}$ is sub-gaussian random variable, since according 
to \eqref{eq: sub-gaussian vector} and  \eqref{eq: ai sub-gaussian}, we have
$\|Z\|_\psitwo \le \|\va\|_\psitwo \le \psi$.
Then, since sub-gaussian distributions have moments of all orders (see \cite[Lemma~5.5]{V tutorial}),
we have $(\E Z^3)^{1/3} \le C_1 \|Z\|_\psitwo \le C_1 \psi$,
where $C_1$ is an absolute constant.
Using this together with isotropy and Cauchy-Schwarz inequality, we obtain
$$
1 = \E Z^2 = \E Z^{1/2} Z^{3/2} 
\le (\E Z)^{1/2} (\E Z^3)^{1/2}
\le (\E Z)^{1/2} (C_1 \psi)^{3/2}.
$$
Squaring both sides implies \eqref{eq: expectation ip sub-gaussian}, 
since we assumed that $\|\vu\|_2 = 1$.

The next steps in the proof of Theorem~\ref{thm: M* general} --
symmetrization and contraction -- go through for sub-gaussian
distributions without change. So \eqref{eq: sym contr applied} is still valid in our case.

Next, the random vector
$$
\vh := \frac{1}{\sqrt{m}} \sum_{i=1}^m \e_i \va_i
$$
is no longer Gaussian as in the proof of Theorem~\ref{thm: M* general}. 
Still, $\vh$ is sub-gaussian with 
\begin{equation}         \label{eq: h sub-gaussian}
\|\vh\|_\psitwo \le C_2 \psi
\end{equation}
due to the approximate rotation invariance of sub-gaussian distributions, 
see \cite[Lemma~5.9]{V tutorial}.

In the last step of the argument, we need to replace the sub-gaussian random vector $\vh$
by the Gaussian random vector $\vg \sim N(0,I_n)$, i.e. prove an inequality of the form
$$
\E \sup_{\vu \in T} |\ip{\vh}{\vu}| \lesssim \E \sup_{\vu \in T} |\ip{\vg}{\vu}|.
$$
This can be done by applying the comparison inequality of 
Theorem~\ref{thm: fernique-talagrand} for the processes
$$
H(\vu) = \ip{\vh}{\vu} \quad \text{and} \quad G(\vu) = \ip{\vg}{\vu}, \quad \vu \in T \cup (-T).
$$ 
To check the increment inequality, we can use \eqref{eq: h sub-gaussian}, which yields
$$
\|H(\vu) - H(\vv)\|_\psitwo 
= \|\ip{\vh}{\vu-\vv}\|_\psitwo
\le \|\vh\|_\psitwo \, \|\vu - \vv\|_2
\le C_2 \psi \, \|\vu - \vv\|_2.
$$
On the other hand, 
$$
(\E \|G(\vu) - G(\vv)\|_2^2)^{1/2} = \|\vu - \vv\|_2.
$$
Therefore, the increment inequality in Theorem~\ref{thm: fernique-talagrand}
holds with $M = C_2 \psi$. It follows that 
$$
\E \sup_{\vu \in T \cup (-T)} \ip{\vh}{\vu} \le C_3 \psi \, \E \sup_{\vu \in T \cup (-T)} \ip{\vg}{\vu}.
$$
This means that
$$
\E \sup_{\vu \in T} |\ip{\vh}{\vu}| \le C_3 \psi \, \E \sup_{\vu \in T} |\ip{\vg}{\vu}|
$$
as claimed.

Replacing all Gaussian inequalities by their sub-gaussian counterparts
discussed above, we complete the proof just like in Theorem~\ref{thm: M* general}.
\end{proof}

\subsection{Estimation from sub-gaussian linear observations}

It is now straightforward to generalize all recovery results we developed 
before from Gaussian to sub-gaussian observations. So our observations are now
$$
y_i = \ip{\va_i}{x} + \vnu_i, \quad i=1,\ldots,m
$$
where $\va_i$ are i.i.d., mean zero, isotropic and {\em sub-gaussian} 
random vectors in $\R^n$. 
As in Theorem~\ref{thm: M* sub-gaussian}, we control the sub-gaussian norm with
the parameter $\psi>1$, choosing it so that 
$$
\|\va_i\|_\psitwo \le \psi, \quad i=1,\ldots,m.
$$
We can write observations in the matrix form 
as in \eqref{eq: noisy observations}, i.e. 
$$
\vy = A\vx + \vnu,
$$
where $A$ is the $m \times n$ matrix with rows $\va_i$. As before,
we assume some control on the error: 
$$
\frac{1}{m} \|\vnu\|_1 
= \frac{1}{m} \sum_{i=1}^m |\nu_i| \le \e.
$$

Let us state a version of Theorem~\ref{thm: estimation noisy feasibility}
for sub-gaussian observations. Its proof is the same, except we use
the sub-gaussian $M^*$ bound, Theorem~\ref{thm: M* sub-gaussian} where 
previously a Gaussian $M^*$ bound was used.

\begin{theorem}[Estimation from sub-gaussian observations]	
  Choose $\xhat$ to be any vector satisfying
  $$  
  \xhat \in K \quad \text{and} \quad \frac{1}{m} \|A\xhat - \vy\|_1 \le \e.
  $$  
  Then 
  $$
  \E \sup_{\vx \in K} \|\xhat-\vx\|_2 \le C \psi^4 \left( \frac{w(K)}{\sqrt{m}} + \e \right).
  \hfill \qed
  $$
\end{theorem}

In a similar fashion, one can generalize all other estimation results 
established before to sub-gaussian observations. We leave this to the 
interested reader.

\section{Exact recovery}					\label{s: exact}

In some situations, one can hope to estimate vector $\vx \in K$ 
from $\vy$ {\em exactly}, without any error. 
Such results form the core of the area of {\em compressed sensing} \cite{DDEK, Kutyniok, FR}.
Here we will present an approach to exact recovery based on Y.~Gordon's ``escape through a mesh'' theorem \cite{Gordon}.
This argument goes back to \cite{RV CPAM} for the set of sparse vectors, it was further developed
in  in \cite{Stojnic, Oymak-Hassibi} and was pushed forward for general feasible sets 
in \cite{CRPW, ALPV, Tropp}. 

In this tutorial we will present
the most basic result; the reader will find a more complete picture and many more examples in 
the papers just cited. 

We will work here with Gaussian observations 
$$
\vy = A\vx,
$$
where $A$ is an $m \times n$ Gaussian random matrix. 
This is the same model as we considered 
in Section~\ref{s: estimation linear}.

\subsection{Exact recovery condition and the descent cone}
 
When can $\vx$ be inferred from $\vy$ exactly? 
Recall that we only know two things about $\vx$ -- that it lies in the feasible set $K$ 
and in the affine subspace 
$$
E_{\vx} : = \{ \vx' : \; A\vx' = \vy \}.
$$
This two pieces of information determine $\vx$ uniquely if and only if these two 
sets intersect at the single point $\vx$:
\begin{equation}         \label{eq: KEx}
K \cap E_{\vx} = \{x\}.
\end{equation}
Notice that this situation would go far beyond the $M^*$ bound
on the diameter of $K \cap E$ (see Theorem~\ref{thm: M*}) -- 
indeed, in this case the diameter would equal zero!

How can this be possible? 
Geometrically, the exact recovery condition \eqref{eq: KEx} 
states that {\em the affine subspace $E_{\vx}$ 
is tangent to the set $K$ at the point $\vx$}; see Figure~\ref{fig: KEx} 
for illustration. 

\begin{figure}[htp]			
  \centering 
  \begin{subfigure}[b]{0.4\textwidth}
    \includegraphics[height=3cm]{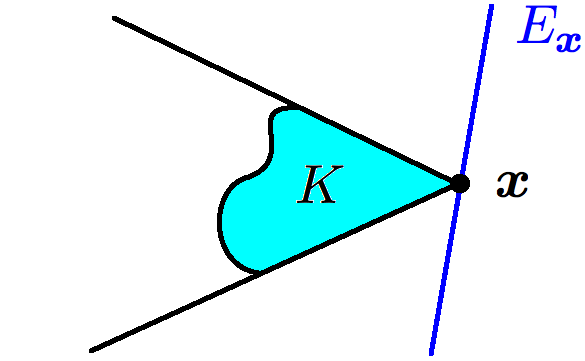} 
    \caption{Exact recovery condition \eqref{eq: KEx}: affine subspace $E_{\vx}$ 
		is tangent to $K$ at $\vx$}
    \label{fig: KEx}
  \end{subfigure}
  \qquad \qquad
  \begin{subfigure}[b]{0.4\textwidth}
    \includegraphics[height=3cm]{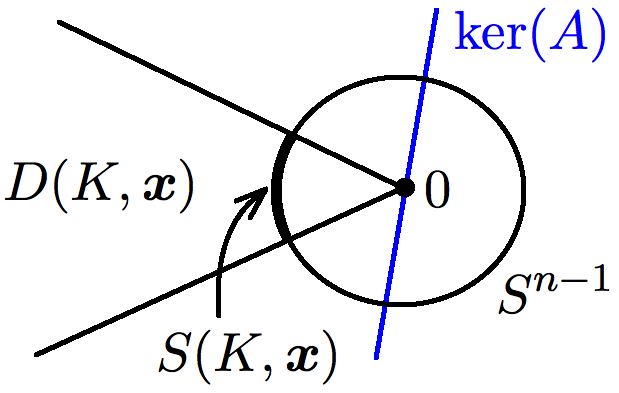} 
    \caption{Picture translated by $-\vx$: subspace $\ker(A)$ 
		is tangent to descent cone $D(K,\vx)$ at $0$}
    \label{fig: descent}
  \end{subfigure}
  \caption{Illustration of the exact recovery condition \eqref{eq: KEx}}	
\end{figure}

This condition is local. Assuming that $K$ is convex
for better understanding, we see that the tangency condition depends on the shape 
of $K$ in an infinitesimal neighborhood of $\vx$, while the global geometry of $K$ 
is irrelevant. So we would not lose anything if we replace $K$ by  
the {\em descent cone} at point $x$, see Figure~\ref{fig: descent}. 
This set is formed by the rays emanating 
from $\vx$ into directions of points from $K$: 
$$
D(K, \vx) := \{ t(\vz-\vx) :\; \vz \in K, \, t \ge 0 \}.
$$

Translating by $-\vx$, can we rewrite the exact recovery condition \eqref{eq: KEx} as 
$$
(K-\vx) \cap (E_{\vx}-\vx) = \{0\}
$$
Replacing $K-\vx$ by the descent cone (a bigger set) and
noting that $E_{\vx}-\vx = \ker(A)$, we rewrite this again as 
$$
D(K, \vx) \cap \ker(A) = \{0\}.
$$
The descent cone can be determined by its intersection with the unit sphere,
i.e. by\footnote{In the definition \eqref{eq: SKx}, we adopt the convention that $0/0=0$.}
\begin{equation}         \label{eq: SKx}
S(K,x) := D(K, \vx) \cap S^{n-1}
= \Big\{ \frac{\vz-\vx}{\|\vz-\vx\|_2} :\; \vz \in K \Big\}.
\end{equation}
Thus we arrive at the following equivalent form of 
the exact recovery condition \eqref{eq: KEx}:
$$
S(K, \vx) \cap \ker(A) = \varnothing;
$$
see Figure~\ref{fig: descent} for an illustration.

\subsection{Escape through a mesh, and implications for exact recovery }

It remains to understand under what conditions 
the random subspace $\ker A$ misses 
a given subset $S = S(K,\vx)$ of the unit sphere. 
There is a remarkably sharp result in asymptotic convex geometry that
answers this question for general subsets $S$. This is the theorem on 
{\em escape through a mesh}, which is due to Y.~Gordon \cite{Gordon}.
Similarly to the other results we saw before, this theorem depends on the 
{\em mean width} of $S$, defined as\footnote{The only (minor) difference 
with our former definition \eqref{eq: mean width} of the mean width  
is that we take supremum over $S$ instead of $S-S$, so $\bar{w}(S)$
is a smaller quantity. The reason we do not need to consider $S-S$ because 
we already subtracted $\vx$ in the definition of the descent cone.}
$$
\bar{w}(S) = \E \sup_{\vu \in S} \ip{\vg}{\vu}, 
\quad \text{where} \quad \vg \sim N(0,I_n).
$$

\begin{theorem}[Escape through a mesh]				\label{thm: Gordon}
  Let $S$ be a fixed subset of $S^{n-1}$.
  Let $E$ be a random subspace of $\R^n$ of a fixed codimension $m$, 
  drawn from the Grassmanian $G_{n,n-m}$ according to the Haar measure. 
  Assume that 
  $$
  \bar{w}(S) < \sqrt{m}.
  $$
  Then 
  $$
  S \cap E = \varnothing
  $$
  with high probability, namely $1 - 2.5 \exp \big[ - (m/\sqrt{m+1} - \bar{w}(S))^2/18 \big]$.
\end{theorem}

Before applying this result to high dimensional estimation, let us see how 
a slightly weaker result follows from the general $M^*$ bound, Theorem~\ref{thm: M* general}.
Indeed, applying the latter theorem for $T=S$, $E=\ker(A)$ and $\e=0$, we obtain 
\begin{equation}         \label{eq: Gordon weaker}
\E \sup_{\vu \in S \cap E} \|\vu\|_2 
  \le \sqrt{\frac{8\pi}{m}} \, \E \sup_{\vu \in S} |\ip{\vg}{\vu}|
  \le \sqrt{\frac{8\pi}{m}} \, \bar{w}(S).
\end{equation}
Since $S \subset S^{n-1}$, the supremum in the left hand side equals $1$ 
when $S \cap E \ne \emptyset$ and zero otherwise. Thus the expectation in \eqref{eq: Gordon weaker}
equals $\Pr{ S \cap E \ne \emptyset }$. Further, one can easily check that
$\E \sup_{\vu \in S} |\ip{\vg}{\vu}| \le \bar{w}(S) + \sqrt{2/\pi}$, see \cite[Proposition~2.1]{PV CPAM}.
Thus we obtain 
$$
\Pr{ S \cap E \ne \emptyset } 
  \le \sqrt{\frac{8\pi}{m}} \Big( \bar{w}(S) + \sqrt{\frac{2}{\pi}} \Big).
$$
In other words, $S \cap E = \emptyset$ with high probability if the codimension 
$m$ is sufficiently large so that $\bar{w}(S) \ll \sqrt{m}$.
Thus we obtain a somewhat weaker form of Escape Theorem~\ref{thm: Gordon}.

\medskip

Now let us apply Theorem~\ref{thm: Gordon} for the descent $S = S(K,x)$ and $E = \ker(A)$. 
We conclude by the argument above that the exact recovery condition \eqref{eq: KEx}
holds with high probability if 
$$
m > \bar{w}(S)^2.
$$

How can we {\em algorithmically} recover $\vx$ in these circumstances? 
We can do the same as in Section~\ref{s: based on M*}, either
using the feasibility program \eqref{eq: feasibility} or, better yet, 
the optimization program \eqref{eq: optimization}. The only difference
is that the diameter of the intersection is now zero, so the recovery is exact. 
The following is an exact version of Theorem~\ref{thm: estimation optimization}.

\begin{theorem}[Exact recovery from linear observations]	\label{thm: exact recovery}
  Choose $\xhat$ to be a solution of the program 
  $$  
  \text{minimize } \|\vx'\|_K \quad \text{subject to} \quad A\vx' = \vy.
  $$  
  Assume that the number of observations satisfies  
  \begin{equation}         \label{eq: m wbar}
  m > \bar{w}(S)^2
  \end{equation}
  where $S = S(K,x)$ is the spherical part of the descent cone of $K$, defined 
  in \eqref{eq: SKx}.
  Then 
  $$
  \xhat = \vx
  $$
  with high probability (the same as in Theorem~\ref{thm: Gordon}).		\qed
\end{theorem}

Note the familiar condition \eqref{eq: m wbar} on $m$ which we 
have seen before, see e.g. Section~\ref{s: algorithmic}.
Informally, it states the following: 
\begin{quote}
{\em Exact recovery is possible when the number of 
measurements exceeds the effective dimension of the descent cone}. 
\end{quote}
Remarkably, the condition \eqref{eq: m wbar} does not have 
absolute constant factors which we had in results before.

\subsection{Application: exact sparse recovery}

Let us illustrate how Theorem~\ref{thm: exact recovery} works 
for {\em exact sparse recovery}. 
Assume that $\vx$ is $s$-sparse, i.e. it has at most $s$ non-zero coefficients.
For the feasible set, we can choose $K := \|\vx\|_1 B_1^n = \{ \vx' :\; \|\vx'\|_1 \le \|\vx\|_1 \}$.
One can write down accurately an expression for the descent cone, 
and derive a familiar bound on the mean width of $S = S(K,x)$: 
$$
\bar{w}(S) \le C \sqrt{s \log (2n/s)}.
$$
This computation goes back to \cite{RV CPAM}; see that paper and also 
\cite{Stojnic, CRPW, ALMT} for estimates with explicit absolute constants.

We plug this into Theorem~\ref{thm: exact recovery}, where we replace
$\|\vx'\|_K$ in the optimization problem by the proportional quantity $\|\vx'\|_1$. 
This leads to the following exact version of Corollary~\ref{cor: sparse recovery}:

\begin{theorem}[Exact sparse recovery]		\label{thm: exact sparse recovery}
  Assume that an unknown vector $\vx \in \R^n$ is $s$-sparse. 
  Choose $\xhat$ to be a solution to the convex program 
  $$
  \text{minimize } \|\vx'\|_1 \quad \text{subject to} \quad A\vx' = \vy.
  $$
  Assume that the number of observations satisfies  
  $m > C s \log n$. Then 
  $$
  \xhat = \vx
  $$
  with high probability, namely $1 - 3 e^{-m}$.		\qed
\end{theorem}

Due to the remarkable sharpness of Gordon's theorem, one may hope to 
obtain {\em sharp} conditions on the number of observations $m$, without any losses in 
absolute constants. This was done in \cite{DT} for the sparse recovery problem
(using geometry of polytopes rather than Gordon's theorem), and 
more recently in \cite{ALMT} for general feasible cones. The latter paper 
proposes a notion of statistical dimension, which is a close relative of mean width, 
and establishes a variant of Gordon's theorem for statistical dimension.

\section{Low-rank matrix recovery and matrix completion}		\label{s: matrix}

\subsection{Background: matrix norms}
The theory we developed so far concerns estimation of {\em vectors} in $\R^n$.
It should not be surprising that this theory can also be applied for {\em matrices}.
Matrix estimation problems were studied recently, in particular in \cite{CR, CT, KMO, CLMR, Recht}. 

Let us recall some basic facts about matrices and their norms. 
We can identify $d_1 \times d_2$ matrices with vectors in $\R^{d_1 \times d_2}$.
The $\ell_2$ norm in $\R^{d_1 \times d_2}$ is then nothing else than
{\em Frobenius} (or Hilbert-Schmidt) norm of matrices: 
$$
\|X\|_F = \Big( \sum_{i=1}^{d_1} \sum_{j=1}^{d_2} |X_{ij}|^2 \Big)^{1/2}.
$$
The inner product in $\R^{d_1 \times d_2}$ can be written in matrix form as follows:
$$
\ip{X}{Y} = \tr(X^\tran Y).
$$

Denote  $d = \min(d_1, d_2)$. 
Let
$$
s_1(X) \ge s_2(X) \ge \cdots \ge s_d(X) \ge 0
$$
denote the {\em singular values} of $X$. Then Frobenius norm 
has the following spectral representation:
$$
\|X\|_F = \Big( \sum_{i=1}^ds_i(X)^2 \Big)^{1/2}.
$$
Recall also the {\em operator norm} of $X$, which is
$$
\|X\| = \max_{\vu \in \R^n \setminus \{0\}} \frac{\|X\vu\|_2}{\|\vu\|_2}
= \max_{i=1,\ldots,d} s_i(X).
$$
Finally, the {\em nuclear norm} of $X$ is defined as
$$
\|X\|_* = \sum_{i=1}^d s_i(X).
$$

Spectrally, i.e. on the level of singular values,
the nuclear norm is a version of $\ell_1$ norm for matrices,
the Frobenius norm is a version of $\ell_2$ norm for matrices,
and the operator norm is a version of $\ell_\infty$ norm for matrices.
In particular, the following inequality holds: 
$$
\|X\| \le \|X\|_F \le \|X\|_*.
$$
The reader should be able to derive many other useful inequalities in a similar 
way, for example
\begin{equation}         \label{eq: nuclear frobenius operator}
\|X\|_* \le \sqrt{\rank(X)} \cdot \|X\|_F, \quad \|X\|_F \le \sqrt{\rank(X)} \cdot \|X\|
\end{equation}
and
\begin{equation}         \label{eq: matrix inner product split}
\ip{X}{Y} \le \|X\| \cdot \|Y\|_*.
\end{equation}

\subsection{Low-rank matrix recovery}

We are ready to formulate a matrix version of the sparse recovery problem 
from Section~\ref{s: sparse recovery}. Our goal is to estimate an unknown 
$d_1 \times d_2$ 
matrix $X$ from $m$ linear observations given by
\begin{equation}         \label{eq: matrix observations}
y_i = \ip{A_i}{X}, \quad i=1,\ldots,m.
\end{equation}
Here $A_i$ are independent $d_1 \times d_2$ Gaussian matrices
with all i.i.d. $N(0,1)$ entries.

There are two natural matrix versions of sparsity. 
The first version is the sparsity of entries. We will be concerned
with the other, spectral, type of sparsity, where there are only a few non-zero 
singular values. This simply means that the matrix has {\em low rank}. 
So let us assume that the unknown matrix $X$ satisfies
\begin{equation}         \label{eq: low rank assumption}
\rank(X) \le r
\end{equation}
for some fixed (and possibly unknown) $r \le n$.

The following is a matrix version of Corollary~\ref{cor: sparse recovery}; 
for simplicity we are stating it in a noise-free setting ($\e=0$).

\begin{theorem}[Low-rank matrix recovery]		\label{thm: low-rank matrix recovery}
  Choose $\Xhat$ to be a solution to the convex program 
  \begin{equation}         \label{eq: minimize nuclear norm}
  \text{minimize } \|X'\|_* \text{ subject to } 
  \ip{A_i}{X'} = y_i, \quad i=1,\ldots,m.
  \end{equation}  
  Then 
  $$
  \E \sup_X \|\Xhat-X\|_F
  \le 4 \sqrt{\pi} \,  \sqrt{\frac{r(d_1+d_2)}{m}} \cdot \|X\|_F.
  $$
  Here the supremum is taken over all $d_1 \times d_2$ matrices $X$ 
  of rank at most $r$.
\end{theorem}

The proof of Theorem~\ref{thm: low-rank matrix recovery} will closely follow 
its vector prototype, that of Theorem~\ref{thm: sparse recovery}; 
we will just need to replace the $\ell_1$ norm by the nuclear norm. 
The only real difference will be in the computation of the {\em mean width of the unit ball 
of the nuclear norm}. This computation will be based on Y.~Gordon's bound 
on the operator norm of Gaussian random matrices, see Theorem~5.32 in \cite{V tutorial}.

\begin{theorem}[Gordon's bound for Gaussian random matrices]		\label{thm: gordon}
  Let $G$ be an $d_1 \times d_2$ matrix whose entries are 
  i.i.d. mean zero random variables. Then
  $$
  \E\|G\| \le \sqrt{d_1} + \sqrt{d_2}.
  $$
\end{theorem}

\begin{proposition}[Mean width of the unit ball of nuclear norm]		\label{prop: nuclear norm wK}
  Consider the unit ball in the space of $d_1 \times d_2$ matrices
  corresponding to the nuclear norm:
  $$
  B_* := \{ X \in \R^{d_1 \times d_2} :\; \|X\|_* \le 1 \}.
  $$
  Then 
  $$
  w(B_*) \le 2 (\sqrt{d_1} + \sqrt{d_2}).
  $$ 
\end{proposition}

\begin{proof}
By definition and symmetry of $B$, we have 
$$
w(B) = \E \sup_{X \in B_*-B_*} \ip{G}{X} = 2 \E \sup_{X \in B_*} \ip{G}{X},
$$
where $G$ is a $d_1 \times d_2$ Gaussian random matrix with $N(0,1)$ entries.
Using inequality \eqref{eq: matrix inner product split} and definition of $B_*$, we obtain
we obtain 
$$
w(B_*) \le 2 \E \sup_{X \in B_*} \|G\| \cdot \|X\|_*
\le 2 \E \|G\|.
$$
(The reader may notice that both these inequalities are in fact equalities, although we do not 
need this in the proof.)
To complete the proof, it remains to apply Theorem~\ref{thm: gordon}.
\end{proof}

Let us mention an immediate consequence of Proposition~\ref{prop: nuclear norm wK},
although it will not be used in the proof of Theorem~\ref{thm: low-rank matrix recovery}.

\begin{proposition}[Mean width of the set of low-rank matrices]		\label{prop: low rank mean width}
  Let 
  $$
  D = \{ X \in \R^{d_1 \times d_2} :\; \|X\|_F = 1, \; \rank(X) \le r \}.
  $$
  Then 
  $$
  w(D) \le 2 \sqrt{2r(d_1+d_2)}.
  $$ 
\end{proposition}

\begin{proof}[Proof of Proposition~\ref{prop: low rank mean width}]
The bound follows immediately from Proposition~\ref{prop: nuclear norm wK}
and the first inequality in \eqref{eq: nuclear frobenius operator}, which implies 
that $D \subset \sqrt{r} \cdot B_*$.
\end{proof}

\medskip

\begin{proof}[Proof of Theorem~\ref{thm: low-rank matrix recovery}]
The argument is a matrix version of the proof of Theorem~\ref{thm: sparse recovery}.
We consider the following subsets of $d_1 \times d_2$ matrices:
$$
\bar{K} := \{ X' :\; \|X'\|_* \le 1\}, \quad K := \|X\|_* \cdot \bar{K}.
$$
Then obviously $X \in K$, so it makes sense to apply 
Theorem~\ref{thm: estimation noisy optimization} (with $\e=0$) for $K$.
It should also be clear that 
the optimization program in Theorem~\ref{thm: estimation noisy optimization}
can be written in the form \eqref{eq: minimize nuclear norm}.

Applying Theorem~\ref{thm: estimation noisy optimization}, we obtain
$$
\E \sup_X \|\Xhat-X\|_F
\le \sqrt{2\pi} \cdot \frac{w(K)}{\sqrt{m}}.
$$
Recalling the definition of $K$ and using Proposition~\ref{prop: nuclear norm wK}
to bound its mean width, we have
$$
w(K) = w(\bar{K}) \cdot \|X\|_* \le 2 \sqrt{2} \,  \sqrt{d_1+d_2} \cdot \|X\|_*.
$$
It follows that 
$$
\E \sup_X \|\Xhat-X\|_F
\le 4 \sqrt{\pi} \, \sqrt{\frac{d_1+d_2}{m}} \cdot \|X\|_*.
$$
It remains to use the low-rank assumption \eqref{eq: low rank assumption}.
According to the first inequality in \eqref{eq: nuclear frobenius operator},
we have 
$$
\|X\|_* \le \sqrt{r} \|X\|_F.
$$
This completes the proof of Theorem~\ref{thm: low-rank matrix recovery}.
\end{proof}

\subsection{Low-rank matrix recovery: some extensions}

\subsubsection{From exact to effective low rank}
The exact low rank assumption \eqref{eq: low rank assumption} can be replaced 
by approximate low rank assumption. This is a matrix version of a similar
observation about sparsity which we made in Section~\ref{s: effective sparsity}.
Indeed, our argument shows that Theorem~\ref{thm: low-rank matrix recovery}
will hold if we replace the rank by the more flexible {\em effective rank}, defined 
for a matrix $X$ as
$$
r(X) = (\|X\|_*/\|X\|_F)^2.
$$
The effective rank is clearly bounded by the algebraic rank, and it is robust 
with respect to small perturbations.

\subsubsection{Noisy and sub-gaussian observations}

Our argument makes it easy to allow noise in the observations \eqref{eq: matrix observations}, 
i.e. consider observations of the form $y_i = \ip{A_i}{X} + \nu_i$. 
We leave details to the interested reader. 

\medskip

Further, just like in Section~\ref{s: sub-gaussian}, we can relax the requirement 
that $A_i$ be Gaussian random matrices, replacing it with a {\em sub-gaussian} 
assumption. Namely, it is enough to assume that the columns of $A_i$ 
are i.i.d., mean zero, isotropic and sub-gaussian random vectors in $\R^{d_1}$, 
with a common bound on the sub-gaussian norm. 
We again leave details to the interested reader.

\medskip

We can summarize the results about low-rank matrix recovery as follows.
\begin{quote}
  {\em Using convex programming, one can approximately recover 
  a $d_1 \times d_2$ matrix which has rank (or effective rank) $r$,
  from $m \sim r (d_1 + d_2)$ random linear observations.}
\end{quote}

To understand this number of observations better, 
note that it is of the same order as the number of degrees of 
freedom in the set of $d_1 \times d_2$ matrices or rank $r$.

\subsection{Matrix completion}				\label{s: matrix completion}

Let us now consider a different, and perhaps more natural, model of observations of matrices. 
Assume that we are given a {\em small random sample of entries} of an unknown matrix 
matrix $X$.
Our goal is to estimate $X$ from this sample. 
As before, we assume that $X$ has low rank.
This is called a {\em matrix completion problem}, and it was extensively 
studied recently \cite{CR, CT, KMO, Recht}.

The theory we discussed earlier in this chapter does not apply here. 
While sampling of entries is a linear operation, such observations are 
not Gaussian or sub-gaussian (more accurately, we should say 
that the sub-gaussian norm of such observations is too large). 
Nevertheless, it is possible able to derive a matrix completion result in this setting. 
Our exposition will be based on a direct and simple argument from \cite{PVY}.
The reader interested in deeper understanding of the matrix completion problem 
(and in particular exact completion) is referred to the papers cited above.

Let us formalize the process of sampling the entries of $X$.
First, we fix the average size $m$ of the sample. Then we 
generate selectors $\d_{ij} \in \{0,1\}$ for each entry of $X$. Those 
are i.i.d. random variables with 
$$
\E \d_{ij} = \frac{m}{d_1 d_2} =: p.
$$
Our observations are given as the $d_1 \times d_2$ matrix $Y$
whose entries are 
$$
Y_{ij} = \d_{ij} X_{ij}.
$$
Therefore, the observations are randomly and independently sampled entries
of $X$ along with the indices of these entries; the average sample size is fixed and equals $m$.
We will require that
\begin{equation}         \label{eq: m d1 d2}
m \ge d_1 \log d_1, \quad m \ge d_2 \log d_2.
\end{equation}
These restrictions ensure that, with high probability, the sample contains 
at least one entry from each row and each column of $X$ 
(recall the classical coupon collector's problem).

As before, we assume that
$$
\rank(X) \le r.
$$
The next result shows that $X$ can be estimated from $Y$ using low-rank approximation.  

\begin{theorem}[Matrix completion]				\label{thm: matrix completion}
  Choose $\Xhat$ to be best rank-$r$ approximation\footnote{Formally, 
  consider the singular value decomposition $p^{-1} Y = \sum_i s_i \vu_i \vv_i^\tran$
  with non-increasing singular values $s_i$. 
  We define $\Xhat$ by retaining the $r$ leading terms of this decomposition, i.e. 
  $\Xhat = \sum_{i=1}^r s_i \vu_i \vv_i^\tran$.} 
  of $p^{-1} Y$. Then 
  \begin{equation}         \label{eq: matrix completion}
  \E \frac{1}{\sqrt{d_1 d_2}} \, \|\Xhat - X\|_F \le C \sqrt{\frac{r(d_1+d_2)}{m}} \, \|X\|_\infty,
  \end{equation}
  where $\|X\|_\infty = \max_{i,j} |X_{ij}|$. 
\end{theorem}

To understand the form of this estimate, note that the left side of \eqref{eq: matrix completion}
measures the {\em average error per entry} of $X$:
$$
\frac{1}{\sqrt{d_1 d_2}} \|\Xhat - X\|_F 
= \Big( \frac{1}{d_1 d_2} \sum_{i=1}^{d_1} \sum_{j=1}^{d_2} |\Xhat_{ij} - X_{ij}|^2 \Big)^{1/2}.
$$
So, Theorem~\ref{thm: matrix completion} allows to make the average error per entry 
arbitrarily smaller than the maximal entry of the matrix. 
Such estimation succeeds with a sample of $m \sim r(d_1 + d_2)$ entries 
of $X$.

\medskip

The proof of Theorem~\ref{thm: matrix completion} will be based on a known
bound on the operator norm of random matrices, which is more general 
than Y.~Gordon's Theorem~\ref{thm: gordon}. 
There are several ways to obtain general bounds; see \cite{V tutorial} 
for a systematic treatment of this topic. We will use one such result 
due to Y.~Seginer \cite{Seginer}.

\begin{theorem}[Seginer's bound for general random matrices]		\label{thm: seginer}
  Let $G$ be an $d_1 \times d_2$ matrix whose entries are 
  i.i.d. mean zero random variables. Then
  $$
  \E\|G\| \le C \Big( \E \max_i \|G_i\|_2 + \E \max_j \|G^j\|_2 \Big)
  $$
  where the maxima are taken over all rows $G_i$ and over all columns $G^j$ of $G$,
  respectively.
\end{theorem}

\medskip

\begin{proof}[Proof of Theorem~\ref{thm: matrix completion}.]
We shall first control the error in the operator norm.
By triangle inequality,
\begin{equation}         \label{eq: via Y}
\|\Xhat - X\| \le \|\Xhat - p^{-1}Y\| + \|p^{-1}Y - X\|.
\end{equation}
Since $\Xhat$ is the best rank-$r$ approximation to $p^{-1}Y$, and 
both $X$ and $\Xhat$ are rank-$r$ matrices, the 
first term in \eqref{eq: via Y} is bounded by the second term. Thus
\begin{equation}         \label{eq: Xhat-X}
\|\Xhat - X\| \le 2 \|p^{-1}Y - X\| = \frac{2}{p} \|Y - pX\|.
\end{equation}

The matrix $Y-pX$ has independent mean zero entries, namely
$$
(Y-pX)_{ij} = (\d_{ij}-p) X_{ij}.
$$
So we can apply Y.~Seginer's Theorem~\ref{thm: seginer}, which yields
\begin{equation}         \label{eq: norm Y-pX}
\E\|Y-pX\| \le C \Big( \E \max_{i \le d_1} \|(Y-pX)_i\|_2 + \E \max_{j \le d_2} \|(Y-pX)^j\|_2 \Big).
\end{equation}

It remains to bound the $\ell_2$ norms of rows and columns of $Y-pX$.
Let us do this for rows; a similar argument would control the columns.
Note that
\begin{equation}         \label{eq: rows norms}
\|(Y-pX)_i\|_2^2 = \sum_{j=1}^{d_2} (\d_{ij}-p)^2 |X_{ij}|^2
\le \sum_{j=1}^{d_2} (\d_{ij}-p)^2 \cdot \|X\|_\infty^2,
\end{equation}
where $\|X\|_\infty = \max_{i,j} |X_{ij}|$ is the $\ell_\infty$ norm of $X$
considered as a vector in $\R^{d_1 \times d_2}$.
To further bound the quantity in \eqref{eq: rows norms} we can use 
concentration inequalities for sums of independent random variables. 
In particular, we can use Bernstein's inequality (see \cite{BBL}), which yields 
$$
\Pr{\sum_{j=1}^{d_2} (\d_{ij}-p)^2 > p d_2 t} \le \exp(-c p d_2 t), \quad t \ge 2.
$$
The first restriction in \eqref{eq: m d1 d2} guarantees that $p d_2 \ge \log d_1$.
This enables us to use the union bound over $i \le d_1$, which yields
$$
\E \max_{i \le d_1} \Big[ \sum_{j=1}^{d_2} (\d_{ij}-p)^2 \Big]^{1/2} \le C_1 \sqrt{p d_2}.
$$
This translates into the following bound for the rows of $Y-pX$:
$$
\E \max_{i \le d_1} \|(Y-pX)_i\|_2 \le C_1 \sqrt{p d_2} \, \|X\|_\infty.
$$

Repeating this argument for columns and putting the two bounds 
into \eqref{eq: norm Y-pX}, we obtain
$$
\E\|Y-pX\| \le C_2 \sqrt{p(d_1+d_2)} \, \|X\|_\infty.
$$
Substituting into \eqref{eq: Xhat-X}, we conclude that
\begin{equation}         \label{eq: error in operator norm}
\E\|\Xhat-X\| \le C_3 \sqrt{\frac{d_1+d_2}{p}} \, \|X\|_\infty.
\end{equation}

It remains to pass to the Frobeinus norm. This is where we use 
the low rank assumption on $X$.
Since both $X$ and $\Xhat$ have ranks bounded by $r$, we have
$\rank(\Xhat-X) \le 2r$. Then,
according to the second inequality in \eqref{eq: nuclear frobenius operator}, 
$$
\|\Xhat-X\|_F \le \sqrt{2r} \, \|\Xhat-X\|.
$$
Combining this with \eqref{eq: error in operator norm}
and recalling that $p = m/(d_1d_2)$ by definition, we
arrive at the desired bound \eqref{eq: matrix completion}.
\end{proof}

\begin{remark}[Noisy observations]
  One can easily extend Theorem~\ref{thm: matrix completion} for noisy sampling, 
  where every observed entry of $X$ is independently corrupted by 
  a mean-zero noise. Formally, we assume that the entries of the observation matrix $Y$ 
  are
  $$
  Y_{ij} = \d_{ij} (X_{ij} + \nu_{ij})
  $$
  where $\nu_{ij}$ are independent and mean zero random variables. 
  Let us further assume that $|\nu_{ij}| \le M$ almost surely. 
  Then a slight modification of the proof of Theorem~\ref{thm: matrix completion} 
  yields the following error bound:
  $$
  \E \frac{1}{\sqrt{d_1 d_2}} \, \|\Xhat - X\|_F \le C \sqrt{\frac{r(d_1+d_2)}{m}} 
    \, \big( \|X\|_\infty + M \big).
  $$
  We leave details to the interested reader.  
\end{remark}

\section{Single-bit observations via hyperplane tessellations}	\label{s: single bit tess}

It may perhaps be surprising that a theory of similar strength 
can be developed for estimation problems with {\em non-linear} observations, 
in which the observation vector $\vy \in \R^m$ depends non-linearly 
on the unknown vector $\vx \in \R^n$.

In this and next sections we explore an example of extreme non-linearity -- 
the one given by the sign function. In Section~\ref{s: general non-linear}, 
we will extend the theory to completely general non-linearities.

\subsection{Single-bit observations}				\label{s: single-bit observations}

As before, our goal is to estimate an unknown vector $\vx$ that lies in a known
feasible set $K \subset \R^n$, from a random observation vector 
$\vy = (y_1,\ldots,y_m) \in \R^m$.
This time, we will work with {\em single-bit observations} $y_i \in \{-1,1\}$.
So, we assume that 
\begin{equation}         \label{eq: single-bit}
y_i = \sign \ip{\va_i}{\vx}, \quad i=1, \ldots,m,
\end{equation}
where $\va_i$ are standard Gaussian random vectors, i.e. $\va_i \sim N(0,I_n)$.
We can represent the model in a matrix form:
$$
\vy = \sign(A \vx),
$$
where $A$ is an $m \times n$ Gaussian random matrix with rows $\va_i$, 
and where our convention is that the sign function is applied to each coordinate
of the vector $A\vx$.

The single-bit model represents an extreme {\em quantization} of the 
linear model we explored before, where $\vy=A\vx$. Only one bit is retained
from each linear observation $y_i$. Yet we hope to estimate $\vx$ as accurately
as if all bits were available.

The model of single-bit observations was first studied in this context
in \cite{BB}. Our discussion will follow \cite{PV DCG}.

\subsection{Hyperplane tessellations}			\label{s: tessellations}

Let us try to understand single-bit observations $y_i$ from a geometric perspective. 
Each $y_i \in \{-1,1\}$ represents the orientation of the vector $\vx$ with respect
to the hyperplane with normal $\va_i$. There are $m$ such hyperplanes. 
The observation vector $\vy = (y_1,\ldots,y_m)$ represents orientation of $\vx$
with respect to all these hyperplanes. 

Geometrically, the $m$ hyperplanes induce a {\em tessellation} of $\R^n$
by {\em cells}. A cell is a set of points that have the same
orientation with respect to all hyperplanes; see Figure~\ref{fig: tessellation-cell}.
Knowing $\vy$ is the same as knowing the cell where $\vx$ lies. 

\begin{figure}[htp]			
  \centering \includegraphics[height=3.2cm]{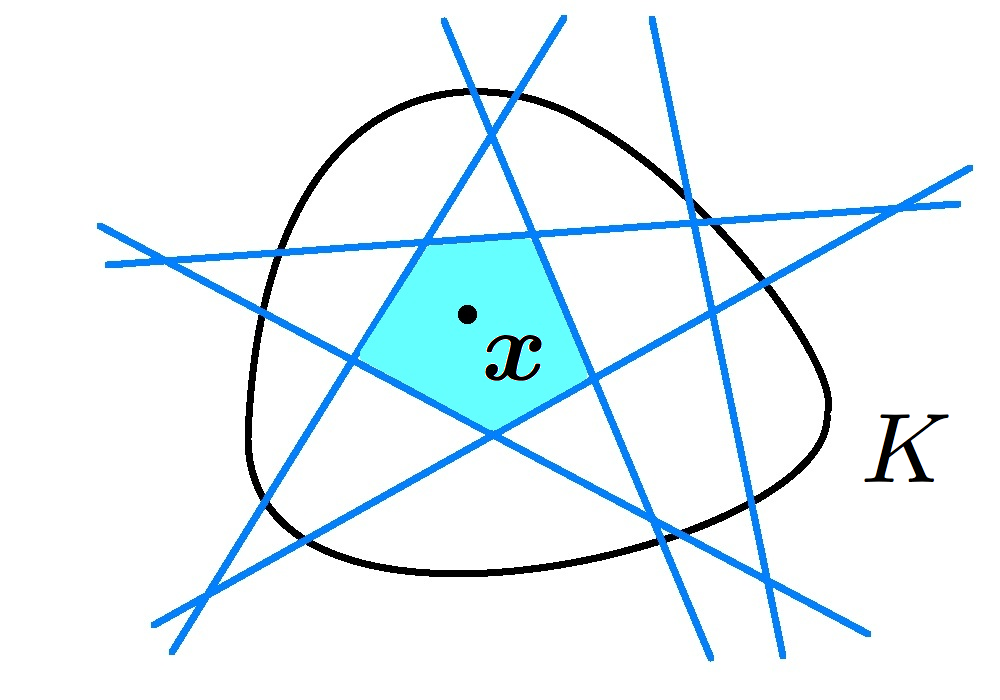} 
  \caption{A tessellation of the feasible set $K$ by hyperplanes. 
    The cell containing $\vx$ is highlighted.}
  \label{fig: tessellation-cell}	
\end{figure}

How can we estimate $\vx$? 
Recall that we know two pieces of information about $\vx$: 
\begin{enumerate}[\qquad 1.]
  \item $\vx$ lies in a known cell of the hyperplane tessellation;
  \item $\vx$ lies in a known set $K$. 
\end{enumerate}

Therefore, a good estimator of $\vx$ can be obtained by picking any 
vector $\xhat$ from the {\em intersection of these two sets}.
Moreover, since just these
two pieces of information about $\vx$ are available, 
such an estimator is best possible in some sense.

\subsection{$M^*$ bound for random tessellations}

How good is such an estimate? The maximal error is of course the diameter of the 
intersection of the cell with $K$. So in order to bound the error, we
need to prove that this diameter is small. 

Note that our strategy is parallel to what we have done for linear observations
in Section~\ref{s: based on M*}. The only piece we are missing is 
a version of $M^*$ bound for random tessellations instead of random subspaces.
Informally, we need a result about the following question: 

\begin{question}[Pizza cutting]
  How many random hyperplanes would cut a given set $K$ into pieces
  that are at most $\e$ in size?
\end{question} 

A result about this problem was proved in \cite{PV DCG}.

\begin{theorem}[$M^*$ bound for random tessellations]	\label{thm: M* tessellations}
  Consider a set $K \subseteq S^{n-1}$ and $m$ independent random hyperplanes
  drawn uniformly from the Grassmanian $G_{n,n-1}$. 
  Then 
  \begin{equation}         \label{eq: M* tessellations}
  \E \max_{\CC} \diam(K \cap \CC) \le \Big[ \frac{C w(K)}{\sqrt{m}} \Big]^{1/3},
  \end{equation}
  where the maximum is taken over all cells $\CC$ of the hyperplane tessellation.\footnote{A 
    high-probability version of Theorem~\ref{thm: M* tessellations} was proved in \cite{PV DCG}. 
    Namely, denoting 
    by $\d$ the right hand side of \eqref{eq: M* tessellations}, we have 
    $\max_{\CC} \diam(K \cap \CC) \le \delta$ with probability at least $1-2 \exp(-c\d^2 m)$,
    as long as $m \ge C \d^{-6} w(K)^2$. The reader will easily deduce the statement 
    of Theorem~\ref{thm: M* tessellations} from this.}
\end{theorem}

Apart from the exponent $1/2$ which is unlikely to be optimal, 
this result is indeed a version of the $M^*$ bound, Theorem~\ref{thm: M*}.
To further highlight the similarity, note that when $m < n$, the intersection 
of the $m$ random hyperplanes is a random linear subspace $E$ of codimension $m$. 
This subspace lies in each cell of the tessellation. So in particular, 
Theorem~\ref{thm: M* tessellations} controls the quantity $\E \diam(K \cap E)$
appearing in the standard $M^*$ bound, Theorem~\ref{thm: M*}.

\subsection{Estimation based on $M^*$ bound for random tessellations}

Now we can apply Theorem~\ref{thm: M* tessellations} for the estimation 
problem. Based on our discussion in Section~\ref{s: tessellations}, 
this result immediately implies the following.

\begin{theorem}[Estimation from single-bit observations: feasibility program]	
		\label{thm: estimation single-bit feasibility}
  Assume the unknown vector $\vx$ lies in some known set $K \subseteq S^{n-1}$, 
  and the single-bit observation vector $\vy$ is given by \eqref{eq: single-bit}.
  Choose $\xhat$ to be any vector satisfying
  \begin{equation}         \label{eq: single-bit feasibility}
  \xhat \in K \quad \text{and} \quad \sign(A\xhat) = \vy.
  \end{equation}
  Then 
  $$
  \E \sup_{\vx \in K} \|\xhat-\vx\|_2 \le  \Big[ \frac{C w(K)}{\sqrt{m}} \Big]^{1/3}.  \quad \qed
  $$
\end{theorem}

We assumed in this result that feasible set $K$ lies on the unit sphere.
This is because the magnitude $\|\vx\|_2$ is obviously lost in the single-bit observations. 
So we can only hope to estimate the direction of $\vx$, which is the vector 
$\vx/\|\vx\|_2$ on the unit sphere.

A good news is that estimation can be made from $m \sim w(K)^2$ single-bit 
observations, the same as for linear observations. So, perhaps surprisingly, 
the essential information about $\vx$ is contained in a single bit of each observation.

A bad news is that the feasibility program \eqref{eq: single-bit feasibility} 
is {\em not convex}. When $K$ is restricted to lie on the sphere, it can 
never be convex or be convexified.
One can get around this issue, for example, by lifting the restriction;
see \cite{PV DCG} for pizza-cutting of general sets in $\R^n$. 

But a better idea will be to replace the feasibility problem \eqref{eq: single-bit feasibility}
by an optimization problem -- just like we did in Section~\ref{s: optimization} --
which will work for general sets $K$ in the unit {\em ball} $B_2^n$ rather than 
the unit sphere. Such sets can be convexified. We will do this in the next section.

\section{Single-bit observations via optimization, and applications to logistic regression}	\label{s: single bit opt}

Our goal remains the same as we described in Section~\ref{s: single-bit observations}. 
We would like to estimate a vector $\vx$ that lies in a known feasible set $K \subset \R^n$,
from single-bit observations given as 
$$
\vy = \sign(A\vx) \in \{-1,1\}^m.
$$

Instead of formulating estimation as a feasibility problem \eqref{eq: single-bit feasibility}, 
we will now state it as an {\em optimization} problem, as follows: 
\begin{equation}         \label{eq: maximize correlation}
\text{maximize } \ip{A\vx'}{\vy} \text{ subject to } 
  \vx' \in K.
\end{equation} 
This program tries to fit linear observations $A\vx'$ to the single-bit observations
$\vy$. It does so by maximizing the correlation between linear and single-bit observations 
while searching inside the feasible set $K$.

If $K$ is a convex set, \eqref{eq: maximize correlation} is a convex program. 
Otherwise one can convexify $K$ as we did several times before.

The following result from \cite{PV IEEE} provides a guarantee for such estimator.

\begin{theorem}[Estimation from single-bit observations: optimization program]	
		\label{thm: estimation single-bit optimization}
  Assume the unknown vector $\vx \in \R^n$ satisfies $\|\vx\|_2 = 1$ 
  and $\vx$ lies in some known set $K \subseteq B_2^n$.
  Choose $\xhat$ to be a solution to the program \eqref{eq: maximize correlation}.
  Then 
  $$
  \E \|\xhat-\vx\|_2^2 \le \frac{C w(K)}{\sqrt{m}}.
  $$
  Here $C = \sqrt{8\pi} \approx 5.01$.
\end{theorem}

Our proof of Theorem~\ref{thm: estimation single-bit optimization}
will be based on properties of the {\em loss function}, which we define as 
$$
L_{\vx}(\vx') = - \frac{1}{m} \ip{A\vx'}{\vy} 
= - \frac{1}{m} \sum_{i=1}^m y_i \ip{\va_i}{\vx'}.
$$
The index $\vx$ indicates that the loss function depends on $\vx$ through $\vy$.
The negative sign is chosen so that program \eqref{eq: maximize correlation} 
minimizes the loss function over $K$.

We will now compute the expected value and the deviation of the loss function 
for fixed $\vx$ and $\vx'$.

\begin{lemma}[Expectation of loss function]				\label{lem: expectation loss}
  Let $\vx \in S^{n-1}$ and $\vx' \in \R^n$. Then 
  $$
  \E L_{\vx}(\vx') = - \sqrt{\frac{2}{\pi}} \, \ip{\vx}{\vx'}.
  $$
\end{lemma}

\begin{proof}
We have
$$
\E L_{\vx}(\vx') = - \E y_1 \ip{\va_1}{\vx'} = - \E \sign(\ip{\va_1}{\vx}) \ip{\va_1}{\vx'}.
$$
It remains to note that $\ip{\va_1}{\vx}$ and $\ip{\va_1}{\vx'}$
are normal random variables with zero mean, 
variances $\|\vx\|_2^2 = 1$ and $\|\vx'\|_2^2$ respectively, 
and covariance $\ip{\vx}{\vx'}$.
A simple calculation renders the expectation above as 
$- \ip{\vx}{\vx'} \cdot \E \sign(g) g$ where $g \sim N(0,1)$.
It remains to recall that $\E \sign(g) g = \E |g| = \sqrt{2/\pi}$.
\end{proof}

\begin{lemma}[Uniform deviation of loss function]				\label{lem: deviation loss}
  We have 
  \begin{equation}         \label{eq: deviation loss}
  \E \sup_{\vu \in K-K} |L_{\vx}(\vu) - \E L_{\vx}(\vu)| 
  \le \frac{2 w(K)}{\sqrt{m}}.
  \end{equation}
\end{lemma}

\begin{proof}
Due to the form of loss function, we can apply the symmetrization inequality 
of Proposition~\ref{prop: symmetrization contraction}, which bounds 
the left side of \eqref{eq: deviation loss} by 
\begin{equation}         \label{eq: big ip}
\frac{2}{m} \E \sup_{\vu \in K-K} \Big| \sum_{i=1}^m \e_i y_i \ip{\va_i}{\vu} \Big|
= \frac{2}{m} \E \sup_{\vu \in K-K} \Big| \ip{\sum_{i=1}^m \e_i y_i \va_i}{\vu} \Big|.
\end{equation}

By symmetry and since $y_i \in \{-1,1\}$, 
the random vectors $\{\e_i y_i \va_i\}$ are distributed identically with $\{\va_i\}$.
In other words, we can remove $\e_i y_i$ from \eqref{eq: big ip} without changing
the value of the expectation. 

Next, by rotation invariance, $\sum_{i=1}^m \va_i$ is distributed identically with 
$\sqrt{m} \, \vg$, where $\vg \sim N(0,I_n)$. Therefore, the quantity 
in \eqref{eq: big ip} equals
$$
\frac{2}{\sqrt{m}} \E \sup_{\vu \in K-K} |\ip{\vg}{\vu}| = \frac{2 w(K)}{\sqrt{m}}.
$$
This completes the proof.
\end{proof}

\medskip

\begin{proof}[Proof of Theorem~\ref{thm: estimation single-bit optimization}.]
Fix $\vx' \in K$. Let us try to bound $\|\vx-\vx'\|_2$ in terms of $L_{\vx}(\vx) - L_{\vx}(\vx')$.
By linearity of the loss function, we have
\begin{equation}         \label{eq: expectation+deviation}
L_{\vx}(\vx) - L_{\vx}(\vx') = L_{\vx}(\vx-\vx') = \E L_{\vx}(\vx-\vx') + D_{\vx}
\end{equation}
where the deviation
$$
D_{\vx} := \sup_{\vu \in K-K} |L_{\vx}(\vu) - \E L_{\vx}(\vu)|
$$
will be controlled using Lemma~\ref{lem: deviation loss} a bit later.

To compute the expected value in \eqref{eq: expectation+deviation}, 
we can use Lemma~\ref{lem: expectation loss}
along with the conditions $\|\vx\|_2 = 1$, $\|\vx'\|_2 \le 1$ 
(the latter holds since $\vx' \in K \subseteq B_2^n$).
This way we obtain
$$
\E L_{\vx}(\vx-\vx') = - \sqrt{\frac{2}{\pi}} \, \ip{\vx}{\vx-\vx'}
\le - \frac{1}{2} \sqrt{\frac{2}{\pi}} \, \|\vx-\vx'\|_2^2.
$$
Putting this into \eqref{eq: expectation+deviation}, we conclude that
\begin{equation}         \label{eq: Lx bounded}
L_{\vx}(\vx) - L_{\vx}(\vx') 
\le - \frac{1}{\sqrt{2\pi}} \, \|\vx-\vx'\|_2^2 + D_{\vx}.
\end{equation}

This bound holds for any fixed $\vx' \in K$ and for any point
in the probability space (i.e. for any realization of the random variables
appearing in this bound).
Therefore \eqref{eq: Lx bounded} must hold for the random vector 
$\vx' = \xhat$, again for any point in the probability space.

The solution $\xhat$ was chosen to minimize the loss function, thus
$L_{\vx}(\xhat) \le L_{\vx}(\vx)$. This means that for $\vx' = \xhat$, 
the left hand side of \eqref{eq: Lx bounded} is non-negative. 
Rearranging the terms, we obtain 
$$
\|\vx-\xhat\|_2^2 \le \sqrt{2\pi} \, D_{\vx}.
$$
It remains to take expectation of both sides and use Lemma~\ref{lem: deviation loss}. 
This yields
$$
\E \|\vx-\xhat\|_2^2 \le \sqrt{2\pi} \, \frac{2 w(K)}{\sqrt{m}}.
$$
This completes the proof of Theorem~\ref{thm: estimation single-bit optimization}.
\end{proof}

\subsection{Single-bit observations with general non-linearities}

The specific non-linearity of observations that we considered so far 
-- the one given by sign function -- did not play a big role in our argument 
in the last section. The same argument, and surprisingly, the 
same optimization program \eqref{eq: maximize correlation}, can serve
any non-linearity in the observations.

So let us consider a general model of single-bit observations 
$\vy = (y_1,\ldots,y_m) \in \{-1,1\}^m$, which satisfy 
\begin{equation}         \label{eq: single-bit again}
\E y_i = \theta(\ip{\va_i}{\vx}), \quad i=1,\ldots,m
\end{equation}
Here $\theta: \R \to \R$ is some {\em link function}, which 
describes non-linearity of observations. 
We assume that $y_i$ are independent given $\va_i$, 
which are standard Gaussian random vectors as before.
The matrix form of this model can be written as
$$
\E \vy = \theta(A \vx),
$$
where $A$ is an $m \times n$ Gaussian random matrix with rows $\va_i$, 
and where our convention is that the $\theta$ is applied to each coordinate
of the vector $A\vx$.

To estimate $\vx$, an unknown vector in a known feasible set $K$, 
we will try to use the same optimization program \eqref{eq: maximize correlation} 
in the last section. This may be surprising since {\em the program does not 
even need to know the non-linearity $\theta$}, nor does it attempt to estimate $\theta$.
Yet, this idea works in general as nicely as for the specific sign function. 
The following result from \cite{PV IEEE} is a general version of 
Theorem~\ref{thm: estimation single-bit optimization}.

\begin{theorem}[Estimation from single-bit observations with general non-linearity]	
		\label{thm: single-bit general non-linearity}
  Assume the unknown vector $\vx \in \R^n$ satisfies $\|\vx\|_2 = 1$ 
  and $\vx$ lies in some known set $K \subseteq B_2^n$.
  Choose $\xhat$ to be a solution to the program \eqref{eq: maximize correlation}.
  Then 
  $$
  \E \|\xhat-\vx\|_2^2 \le \frac{4 w(K)}{\l \sqrt{m}}.
  $$
  Here we assume that 
  \begin{equation}         \label{eq: lambda}
  \l := \E \theta(g) g > 0 \quad \text{for } g \sim N(0,1).
  \end{equation}
\end{theorem}

\begin{proof}
The argument follows very closely the proof of Theorem~\ref{thm: estimation single-bit optimization}.
The only different place is the computation 
of expected loss function in Lemma~\ref{lem: expectation loss}.
When the sign function is replaced by a general non-linearity $\theta$,
one easily checks that the expected value becomes
$$
\E L_{\vx}(\vx') = - \lambda \ip{\vx}{\vx'}.
$$
The rest of the argument is the same. 
\end{proof}

For $\theta(z) = \sign(z)$, Theorem~\ref{thm: single-bit general non-linearity}
is identical with Theorem~\ref{thm: estimation single-bit optimization}. However, 
the new result is much more general.
{\em Virtually no restrictions are imposed on the non-linearity $\theta$.} 
In particular, $\theta$ needs not be continuous or one-to-one. 

The parameter $\l$ simply measures the information content retained
through the non-linearity. It might be useful to express $\l$ as 
$$
\l = \E \theta(\ip{\va_i}{\vx}) \ip{\va_i}{\vx},
$$
so $\lambda$ measures how much 
the non-linear observations $\theta(\ip{\va_i}{\vx})$ 
are correlated with linear observations $\ip{\va_i}{\vx}$.

The assumption that $\l>0$ is made for convenience; 
if $\l<0$ we can switch the sign of $\theta$. However, if $\l=0$,
the non-linear and linear measurements are uncorrelated, and 
often no estimation is possible. 
An extreme example of the latter situation occurs when $\theta$ is a constant
function, which clearly carries no information about $\vx$.

\subsection{Logistic regression, and beyond}

  For the link function $\theta(z) = \tanh(z/2)$, the estimation problem \eqref{eq: single-bit again}
  is equivalent to {\em logistic regression with constraints}. In the usual statistical notation
  explained in Section~\ref{s: regression}, logistic regression takes the form
  $$
  \E \vy = \tanh(X \vbeta/2). 
  $$
  The coefficient vector $\beta$ is constrained to lie in some known feasible set $K$.
  We will leave it to the interested reader to translate 
  Theorem~\ref{thm: single-bit general non-linearity} into the language of
  logistic regression, just like we did in Section~\ref{s: regression} for linear regression.
  
  The fact that Theorem~\ref{thm: single-bit general non-linearity} applies
  for general and unknown link function should be important in statistics.
  It means that one {\em does not need to know the non-linearity of
  the model (the link function) to make inference}. Be it the $\tanh$ function specific 
  to logistic regression or (virtually) any other non-linearity, the estimator $\vbetahat$ is 
  the same.

\section{General non-linear observations via metric projection}	
\label{s: general non-linear}

Finally, we pass to the most general model of observations 
$\vy = (y_1,\ldots,y_m)$, which are not necessarily linear or single-bit. 
In fact, we will not even specify a dependence of $y_i$ on $\vx$. 
Instead, we only require that $y_i$ be i.i.d.random variables, and  
\begin{equation}         \label{eq: single-index}
\text{each observation $y_i$ may depend on $\va_i$ only through $\ip{\va_i}{\vx}$.}
\end{equation}
Technically, the latter requirement means that, given $\ip{\va_i}{\vx}$, 
the observation $y_i$ is independent from $\va_i$.
This type of observation models are called {\em single-index models} in statistics.

How can we estimate $\vx \in K$ from such general observation vector $\vy$?
Let us look again at the optimization problem \eqref{eq: maximize correlation}, 
writing it as follows:
$$
\text{maximize } \ip{\vx'}{A^\tran\vy} \text{ subject to } 
  \vx' \in K.
$$
It might be useful to imagine solving this program as a sequence of two steps: 
(a) compute a {\em linear estimate} of $\vx$, which is 
\begin{equation}         \label{eq: linear estimate}
\xlin = \frac{1}{m} A^\tran\vy = \frac{1}{m} \sum_{i=1}^m y_i \va_i,
\end{equation}
and then (b) {\em fitting} $\xlin$ to the feasible set $K$, which is done
by choosing a point in $K$ that is most correlated with $\xlin$.

\medskip

Surprisingly, almost the  same estimation procedure succeeds for the general 
single-index model \eqref{eq: single-index}.
We just need to adjust the second, fitting, step. Instead of maximizing the 
correlation, let us metrically {\em project $\xlin$ onto the feasible set $K$}, 
thus choosing $\xhat$ to be a solution of the program
\begin{equation}         \label{eq: projection}
\text{minimize } \|\vx' - \xlin\|_2 \text{ subject to } \vx' \in K.
\end{equation}
Just like in the previous section, it may be surprising that this estimator
does not need to know the nature of the non-linearity in observations $\vy$.
To get a heuristic evidence of why this knowledge may not be needed, 
one can quickly check (using roration invariance) that
$$
\E \xlin = \E y_1 \va_1 = \lambda \bar{\vx}, 
\quad \text{where} \quad \bar{\vx} = \vx/\|\vx\|_2, \quad \lambda = \E y_1 \ip{\va_1}{\bar{\vx}}.
$$
So despite not knowing the non-linearity, $\xlin$ already provides 
an {\em unbiased estimate} of $\vx$, up to scaling.

\medskip

A result from \cite{PVY} provides a guarantee for the two-step estimator
\eqref{eq: linear estimate}, \eqref{eq: projection}.
Let us state this result in a special case where $K$ is a {\em cone}, i.e. 
$tK = K$ for all $t \ge 0$. A version for general sets $K$ is not much more
difficult, see \cite{PVY} for details. 

Since cones are unbounded sets, the standard mean width (as defined 
in \eqref{eq: mean width}) would be infinite. 
To get around this issue, we should consider a {\em local} version of  mean width, 
which we can define as
$$
w_1(K) = \E \sup_{\vu \in (K-K) \cap B_2^n} \ip{\vg}{\vu}, \quad \vg \sim N(0,I_n).
$$

\begin{theorem}[Estimation from non-linear observations]		\label{thm: non-linear}
  Assume the unknown vector $\vx$ lies in a known closed cone $K$ in $\R^n$.  
  Choose $\xhat$ to be a solution to the program \eqref{eq: projection}.
  Let $\bar{\vx} = \vx/\|\vx\|_2$. Then 
  $$
  \E \xhat = \lambda \bar{\vx}  \quad \text{and} \quad  
  \E \|\xhat - \lambda \bar{\vx}\|_2 \le \frac{M w_1(K)}{\sqrt{m}}.
  $$
  Here we assume that 
  $$
  \lambda = \E y_1 \ip{\va_1}{\bar{\vx}} > 0 \quad \text{and} \quad
  M = \sqrt{2\pi} \big[ \E y_1^2 + \Var \big(y_1 \ip{\va_1}{\bar{\vx}}\big) \big]^{1/2}.
  $$
\end{theorem}

The proof of Theorem~\ref{thm: non-linear} is given in \cite[Theorem~2.1]{PVY}. 
It is not difficult, and is close in spirit to the arguments we saw here; 
we will not reproduce it.

\medskip

The role of parameters $\lambda$ and $M$ is to determine the correct
magnitude and deviation of the estimator; one can think of them as {\em constants}
that are usually easy to compute or estimate. 
By rotation invariance, $\lambda$ and $M$ 
depend on the {\em magnitude} $\|\vx\|_2$ (through $y_1$) 
but not on the direction $\bar{\vx} = \vx/\|\vx\|_2$ of the unknown vector $\vx$. 

We can summarize results of this and previous section as follows.
\begin{quote}
  {\em One can estimate a vector $\vx$ 
  in a general feasible set $K$ from $m \sim w(K)^2$ random non-linear observations,
  even if the non-linearity is not known. If $K$ is convex, estimation can be done
  using convex programming.}
\end{quote}

\subsection{Examples of observations}

To give a couple of concrete examples, consider {\em noisy linear observations} 
$$
y_i = \ip{\va_i}{\vx} + \nu_i.
$$
We already explored this model in Section~\ref{s: noisy}, 
where $\nu_i$ were arbitrary numbers representing noise. 
This time, let us assume $\nu_i$ are independent random variables with 
zero mean and variance $\s^2$. A quick computation gives
$$
\lambda = \|\vx\|_2, \quad M = C (\|\vx\|_2 + \s). 
$$
Theorem~\ref{thm: non-linear} then yields the following error bound: 
$$
\E \|\xhat - \vx\|_2 \le \frac{C w_1(K)}{\sqrt{m}} \, (\|\vx\|_2 + \s).
$$

\medskip

Let us give one more example, for the {\em single-bit observations} 
$$
y_i = \sign \ip{\va_i}{\vx}.
$$
We explored this model in Sections~\ref{s: single bit tess} and \ref{s: single bit opt}.
A quick computation gives
$$
\lambda = \sqrt{\frac{2}{\pi}}, \quad M = C.
$$
Theorem~\ref{thm: non-linear} then yields the following error bound: 
$$
\E \big\|\xhat - \sqrt{\frac{2}{\pi}} \, \vx \big\|_2 \le \frac{C w_1(K)}{\sqrt{m}}.
$$

\subsection{Examples of feasible cones}

To give a couple of concrete examples of feasible cones, consider 
the set $K$ of {\em $s$-sparse vectors} in $\R^n$, those with at most $s$ non-zero 
coordinates. As we already noted in Example~\ref{ex: sparse mean width}, 
$$
w_1(K) \sim \sqrt{s \log (2n/s)}.
$$
Further, solving the program \eqref{eq: projection} (i.e. computing 
the metric projection of $\xlin$ onto $K$) amounts to {\em hard thresholding} of $\vx'$.
The solution $\xhat$ is obtained from $\xlin$ by keeping the $s$ largest
coefficients (in absolute value) and zeroing out all other coefficients. 

So Theorem~\ref{thm: non-linear} in this case can be stated informally as follows:

\begin{quote}
{\em One can estimate an $s$-sparse vector $\vx$ in $\R^n$ from $m \sim s \log n$ 
non-linear observations $\vy$, even if the non-linearity is not known. 
The estimation is given by the hard thresholding of 
$\xlin = m^{-1} A^\tran \vy$.} 
\end{quote}

\medskip

Another popular example of a feasible cone is a set of {\em low-rank matrices}. 
Let $K$ be the set of $d_1 \times d_2$ matrices with rank at most $r$.
Proposition~\ref{prop: low rank mean width} implies that
$$
w_1(K) \le C \sqrt{r(d_1+d_2)}.
$$ 

\medskip 

Further, solving the program \eqref{eq: projection} (i.e. computing the metric projection 
of $\vx'$ onto $K$) amounts to computing the best rank-$r$ approximation of $\xlin$.
This amounts to {\em hard thresholding of singular values} of $\xlin$, 
i.e. keeping the leading $s$ terms of the singular value decomposition. 
Recall that we already came across this thresholding
in the matrix completion problem, Theorem~\ref{thm: matrix completion}.

So Theorem~\ref{thm: non-linear} in this case can be stated informally as follows:

\begin{quote}
{\em One can estimate an $d_1 \times d_2$ matrix with rank $r$ 
from $m \sim r(d_1+d_2)$ non-linear observations, even if the non-linearity is not known. 
The estimation is given by the hard thresholding of singular values of $\xlin$.}
\end{quote}

\section{Some extensions}			\label{s: extensions}

\subsection{From global to local mean width}

As we have seen, the concept of Gaussian mean width captures  
the complexity of a feasible set $K$ quite accurately. Still, it is not 
exactly the optimal quantity in geometric and estimation results. 
An optimal quantity is the {\em local mean width}, which is a function 
of radius $r>0$, defined as
$$
w_r(K) = \E \sup_{\vu \in (K - K) \cap r B_2^n} \ip{\vg}{\vu}, \quad \vg \sim N(0,I_n).
$$
Comparing to Definition~\ref{def: mean width} of the usual mean width, 
we see that 
$$
w_r(K) \le w(K) \quad \text{for all } r.
$$

The usefulness of local mean width was noted in asymptotic convex geometry 
by A.~Giannopoulos and V.~Milman \cite{GM1, GM2, GM3, GM4}. 
They showed that the function $w_r(K)$ 
completely describes the diameter of high dimensional sections $K \cap E$, 
thus proving {\em two-sided} versions of the $M^*$ bound (Theorem~\ref{thm: M*}).
An observation of a similar nature was made recently by S.~Chatterjee \cite{Chatterjee}  
in the context of high dimensional estimation. He noted that a variant of local mean width 
provides optimal error rates for the {\em metric projection} onto a feasible set
considered in Section~\ref{s: general non-linear}.

For most results discussed in this survey, one can be replace the usual mean width by 
a local mean width, thus making them stronger. Let us briefly indicate how this can be 
done for the $M^*$ bound (Theorem~\ref{thm: M*}; see \cite{GM2, GM3, GM4, MPT} for a more detailed
discussion.

Such localization is in a sense automatic; it can be done as a ``post-processing''
of the $M^*$ estimate. The conclusion of the general $M^*$ bound, Theorem~\ref{thm: M* general}, 
for $T \cap r B_2^n$, is that 
\begin{equation}         \label{eq: M* localized}
\sup_{\vu \in T_\e \cap r B_2^n} \|\vu\|_2 
\le C \Big( \frac{1}{\sqrt{m}} \, \E \sup_{\vu \in T \cap r B_2^n} |\ip{\vg}{\vu}| + \e \Big)
\end{equation}
with high probability (see also Section~\ref{s: whp}.)
Let us show that the intersection with the ball $r B_2^n$ can be automatically 
removed from the left side.
Since
$$
\sup_{\vu \in T_\e \cap r B_2^n} \|\vu\|_2 = \min \big( \sup_{\vu \in T_\e} \|\vu\|_2, r \big),
$$
it follows that
if $\sup_{\vu \in T_\e \cap r B_2^n} \|\vu\|_2 < r$ then 
$\sup_{\vu \in T_\e} \|\vu\|_2 \le r$.
Thus, if the right side of \eqref{eq: M* localized} is smaller than $r$, 
then $\sup_{\vu \in T_\e} \|\vu\|_2 \le r$.

When applied to the classical $M^*$ bound, Theorem~\ref{thm: M*}, this argument 
localizes it as follows.
$$
\frac{w_r(K)}{r} \le c \sqrt{m} 
\quad \text{implies} \quad
\diam(K \cap E) \le r 
$$
with high probability.

\subsection{More general distributions}

For simplicity of exposition, the estimation results in this survey were stated for isotropic 
Gaussian vectors $\va_i$. We showed in Section~\ref{s: sub-gaussian} how to extend the $M^*$ bound 
and the corresponding linear estimation results for line for {\em sub-gaussian} distributions. 
For more heavy-tailed distributions, a version of $M^*$ bound was proved recently 
in \cite{Mendelson sections}; compressed sensing for such distributions was examined in \cite{LM1, LM2}.

For single-bit observations of Section~\ref{s: single bit opt}, 
a generalization for sub-gaussian distributions is discussed in \cite{ALPV}.
Some results can be formulated for {\em anisotropic} Gaussian distributions, 
where $\va_i \sim N(0, \Sigma)$ with $\Sigma \ne I_n$, see e.g. \cite[Section~3.4]{PV IEEE}.

Results for extremely heavy-tailed distributions, such as samples of entries and
random Fourier measurements, exist currently only for special cases of feasible sets $K$. 
When $K$ consists of sparse vectors, reconstruction of $\vx$ from Fourier measurements (random frequencies
of $\vx$) was extensively studied in compressed sensing \cite{DDEK, Kutyniok, CGLP, FR}.
Reconstruction of a matrix from a random sample of entries was discussed in 
Section~\ref{s: matrix completion} in the context of matrix completion problem.

There are currently no results, for instance, about reconstruction of $\vx \in K$ 
from random Fourier measurements, where $K$ is a general feasible set. It is clear that 
$K$ needs to be {\em incoherent} with the Fourier basis of exponentials, but this has yet to be quantified. 
In the special case where $K$ is a set of sparse vectors, basic results of compressed sensing
quantify this incoherence via a {\em restricted isometry property} \cite{DDEK, Kutyniok, CGLP, FR}.

\end{document}